\documentclass[11pt]{amsart}

\usepackage[letterpaper,margin=1in]{geometry}
\usepackage{amssymb,amsmath,amsthm,enumerate,tikz}

\usepackage[colorlinks]{hyperref}

\usepgflibrary[arrows]
\usetikzlibrary{arrows,matrix}

\newtheorem{prop}{Proposition}[subsection]
\newtheorem{lem}[prop]{Lemma}
\newtheorem{thrm}[prop]{Theorem}

\theoremstyle{definition}
\newtheorem{defn}[prop]{Definition}
\newtheorem{ex}[prop]{Example}

\theoremstyle{remark}
\newtheorem{rmk}[prop]{Remark}

\newcommand{\F}{\mathcal F}

\newcommand{\g}{\mathfrak g}
\newcommand{\wt}[1]{\widetilde #1}
\newcommand{\red}{\twoheadrightarrow}
\newcommand{\cored}{\rightarrowtail}

\DeclareMathOperator{\Man}{\mathbf{Man}}
\DeclareMathOperator{\Symp}{\mathbf{Symp}}
\DeclareMathOperator{\SRel}{\mathbf{SRel}}
\DeclareMathOperator{\dom}{dom}

\title{Double Groupoids and the Symplectic Category}
\author{Santiago Ca\~nez}
\address{Department of Mathematics, Northwestern University, 2033 Sheridan Rd, Evanston IL 60208, USA}
\email{scanez@northwestern.edu}
\date{}

\begin{document}

\begin{abstract}
We introduce the notion of a \emph{symplectic hopfoid}, which is a ``groupoid-like'' object in the category of symplectic manifolds where morphisms are given by canonical relations. Such groupoid-like objects arise when applying a version of the cotangent functor to the structure maps of a Lie groupoid.  We show that such objects are in one-to-one correspondence with symplectic double groupoids, generalizing a result of Zakrzewski concerning symplectic double groups and Hopf algebra objects in the aforementioned category. The proof relies on the fact that one can realize the core of a symplectic double groupoid as a symplectic quotient of the total space. The resulting constructions apply more generally to give a correspondence between double Lie groupoids and groupoid-like objects in the category of smooth manifolds and smooth relations, and we show that the cotangent functor relates the two constructions.
\end{abstract}

\maketitle\thispagestyle{empty}

\section{Introduction}
The symplectic category is the category-like structure whose objects are symplectic manifolds and morphisms are canonical relations, i.e. Lagrangian submanifolds of products of symplectic manifolds. Although compositions in this ``category'' are only defined between certain morphisms, this concept has nonetheless proved to be useful for describing many constructions in symplectic geometry in a more categorical manner; in particular, symplectic groupoids can be characterized as certain monoid objects in this category, and Hamiltonian actions can be described as actions of such monoids. The importance of representing constructions in symplectic geometry in terms of canonical relations is grounded in the idea that such relations give rise to linear maps after quantization.

Symplectic double groupoids are symplectic manifolds equipped with two compatible symplectic groupoid structures, and first arose in the study of Poisson groupoids and their integration~(\cite{W1},\cite{M}). In particular, they provide a natural setting in which to discuss the duality of Poisson groupoids. Symplectic double groups, i.e. symplectic double groupoids over a point, were shown by Zakrzewski~\cite{SZ1},\cite{SZ2} to be the same as Hopf algebra objects in the symplectic category described above. In this paper, we build on Zakrzewski's results, giving a characterization of general symplectic double groupoids in terms of the symplectic category. The resulting objects are referred to as \emph{symplectic hopfoids}, a term which is meant to be reminiscent of a ``groupoid'' where the coproduct must specified as part of the data. A symplectic hopfoid over a point is in particular a hopf algebra object in the symplectic category.

Characterizing symplectic double groupoids in terms of canonical relations leads to the possiblity of a new approach to quantizing Poisson groupoids. This was carried out to some success in~\cite{SZ} in the special case of Zakrzewski's prior work, and it would be interesting to carry this out in the double groupoid setting using the results of this paper. It would also be interesting to know the connection between our construction and those in~\cite{MT}. We will not address these questions in the present work.

The original motivation for this work was the following. It is well known that to a Lie groupoid $G \rightrightarrows M$ one can apply the tangent functor to all structure maps to obtain the so-called tangent groupoid $TG \rightrightarrows TM$~(\cite{H}). This tangent groupoid has many uses, such as in the construction of tangent stacks. However, since there is no analogous ``cotangent functor'' if we restrict ourselves to the usual category of manifolds and smooth maps, a similar construction cannot be carried out directly in order to produce a cotangent groupoid. This problem is circumvented if we allow general relations as morphisms in the target category of the cotangent functor\footnote{As already mentioned, the resulting structure is not a true category. This will not be an issue in our results, but we mention that there do exist solutions to this problem: work with germs of canonical relations as in~\cite{CDW2}, or work with sequences of canonical relations as in~\cite{W2}. In these settings, the cotangent functor becomes an honest functor.}, and applying this version of the functor to the structure maps of $G \rightrightarrows M$ gives an object we denote by $T^*G \rightrightarrows T^*M$. Weinstein suggested in \cite{W} that such an object should be related to the notion of a cotangent stack, and the present work grew out of a desire to understand the structure encoded by $T^*G \rightrightarrows T^*M$. We will characterize this structure as being that of a symplectic hopfoid. One can attempt to further realize Weinstein's suggestion by restricting the relations appearing in $T^*G \rightrightarrows T^*M$ to the domains on which they give actual funcitons, which in the case where $G \rightrightarrows M$ is an orbifold indeed produces the usual cotangent orbifold; we will leave such considerations to a future paper.

The paper is organized as follows. Section 2 contains background material on smooth relations and the symplectic category. Zakrezwki's prior work is recalled here. Section 3 contains background material on symplectic double groupoids, and derives results which setup their description in terms of the symplectic category. The key result is Theorem~\ref{thrm:core}, which observes that one can obtain the core of a symplectic double groupoid via symplectic reduction. This observation in turn depends on the technical result of Lemma~\ref{lem:fol}, which uses the available double groupoid structure to explicitly describe the characteristic foliations of certain coisotropic submanifolds. For instance, in the case of the standard double groupoid structure on the cotangent bundle $T^*G$ of a Lie groupoid $G \rightrightarrows M$, the coisotropic submanifolds in question are
\[ T^*G|_M,\ N^*\F_s,\ N^*\F_t, \]
where $\F_s$ and $\F_t$ are the foliations of $G$ given by the source and target fibers respectively, whose reductions are indeed symplectomorphic to the core $T^*M$ of the double groupoid $T^*G$. Section 4 introduces the notion of a symplectic hopfoid and shows that symplectic double groupoids produce such structures. The constructions described here apply to double Lie groupoids in general, producing what we call Lie hopfoids, and it is also shown in Theorem~\ref{thrm:lie-hopf} that the cotangent functor makes the diagram
\begin{center}
\begin{tikzpicture}[>=angle 90]
	\node (UL) at (0,1) {double Lie groupoids};
	\node (UR) at (6,1) {symplectic double groupoids};
	\node (LL) at (0,-1) {Lie hopfoids};
	\node (LR) at (6,-1) {symplectic hopfoids};

	\tikzset{font=\scriptsize};
	\draw[->] (UL) to node [above] {$T^*$} (UR);
	\draw[->] (UL) to node [left] {hopf} (LL);
	\draw[->] (UR) to node [right] {hopf} (LR);
	\draw[->] (LL) to node [above] {$T^*$} (LR);
\end{tikzpicture}
\end{center}
commute at the level of objects, where ``hopf'' denotes the map sending double groupoids to hopfoids. The recovery of a double groupoid from a hopfoid is also outlined, and the $1$-to-$1$ correspondence between these structures is summarized in Theorem~\ref{thrm:main-cor}. We finish with a brief discussion of morphisms between hopfoids, but leave proper development of such a concept to later work.

\subsubsection*{Acknowledgements}
I would like to thank Alan Weinstein, Rajan Mehta, and the anonymous referees for helpful comments and suggestions.

\section{Categories of Relations}

\subsection{Smooth Relations}
\begin{defn}
A \emph{smooth relation} $R$ from a manifold $M$ to a manifold $N$ is a closed embedded submanifold of $M \times N$. We will use the notation $R: M \to N$ to mean that $R$ is a smooth relation from $M$ to $N$. We will suggestively use the notation $R: m \mapsto n$ to mean that $(m,n) \in R$. The \emph{transpose} of a smooth relation $R: M \to N$ is the smooth relation $R^t: N \to M$ defined by the condition that $(n,m) \in R^t$ if and only if $(m,n) \in R$.
\end{defn}

We define a composition of smooth relations using the usual composition of relations: given smooth relations $R: M \to N$ and $R': N \to Q$, the composition $R' \circ R: M \to Q$ is
\[
R' \circ R := \{ (m,q) \in M \times Q\ |\ \text{there exists } n \in N \text{ such that } (m,n) \in R \text{ and } (n,q) \in R'\}.
\]
This is the same as taking the intersection of $R \times R'$ and $M \times \Delta_N \times Q$ in $M \times N \times N \times Q$, where $\Delta_N$ denotes the diagonal in $N \times N$, and projecting to $M \times Q$. However, we immediately run into the problem that the above composition need no longer produce a smooth closed submanifold of $M \times Q$. To fix this, we introduce the following notions:

\begin{defn}
A pair $(R,R')$ of smooth relations $R: M \to N$ and $R': N \to Q$ is \emph{transversal} if the submanifolds $R \times R'$ and $M \times \Delta_N \times Q$ intersect transversally in $M \times N \times N \times Q$. The pair $(R,R')$ is \emph{strongly transversal} if it is transversal and in addition the projection of
\[
(R \times R') \cap (M \times \Delta_N \times Q)
\]
to $M \times Q$ is a proper embedding.
\end{defn}

As a consequence, for a strongly transversal pair $(R,R')$, the composition $R' \circ R$ is indeed a smooth relation from $M$ to $Q$.

\begin{defn}
The \emph{domain} of $R: M \to N$ is
\[
\dom R := \{m \in M\ |\ \text{there exists $n \in N$ such that $(m,n) \in R$}\} \subseteq M.
\]
The relation $R: M \to N$ is said to be:
\begin{itemize}
\item \emph{surjective} if for any $n \in N$ there exists $m \in M$ such that $(m,n) \in R$,
\item \emph{injective} if whenever $(m,n), (m',n) \in R$ we have $m=m'$,
\item \emph{cosurjective} if for any $m \in M$ there exists $n \in N$ such that $(m,n) \in R$,
\item \emph{coinjective} if whenever $(m,n), (m,n') \in R$ we have $n=n'$.
\end{itemize}
\end{defn}

Note that $R$ is cosurjective if and only if $R^t$ is surjective and $R$ is coinjective if and only if $R^t$ is injective.

\begin{defn}
A smooth relation $R: M \to N$ is said to be a \emph{surmersion} if it is surjective and coinjective, the projection of $R$ to $M$ is a proper embedding, and the projection of $R$ to $N$ is a submersion; it is a \emph{cosurmersion} if $R^t: N \to M$ is a surmersion.
\end{defn}

It is straightforward to check that $R$ is a surmersion if and only if $R \circ R^t = id$, and hence a cosurmersion if and only if $R^t \circ R = id$. It is also true that a pair $(R,R')$ is always strongly transversal if either $R$ is a surmersion or $R'$ a cosurmersion, see ~\cite{W2}.

\subsection{Canonical Relations}
\begin{defn}
A \emph{canonical relation} $L: P \to Q$ from a symplectic manifold $P$ to a symplectic manifold $Q$ is a smooth relation which is Lagrangian as a submanifold of $\overline{P} \times Q$, where $\overline P$ denotes $P$ with the symplectic form multiplied by $-1$.
\end{defn}

\begin{ex}
The graph of a symplectomorphism $f: P \to Q$ is a canonical relation $P \to Q$, which by abuse of notation we will also denote by $f$. In particular, given any symplectic manifold $P$, the graph of the identity map is the canonical relation $id: P \to P$ given by the diagonal in $\overline P \times P$. More generally, the graph of a symplectic \'etale map is a canonical relation.
\end{ex}

\begin{ex}
For any manifold $M$, the \emph{Schwartz transform} on $T^*M$ is the canonical relation
\[
s: T^*M \to \overline{T^*M},\ (p,\xi) \mapsto (p,-\xi)
\]
given by multiplication by $-1$ in the fibers. Alternatively, it is the Lagrangian submanifold of $T^*M \times T^*M \cong T^*(M \times M)$ given by the conormal bundle to the diagonal of $M \times M$.
\end{ex}

\begin{ex}
For any symplectic manifold $S$, a canonical relation $pt \to S$ or $S \to pt$ is nothing but a closed Lagrangian submanifold of $S$.
\end{ex}

One can check that for a canonical relation $L P \to Q$, if the projection from $L$ to $P$ is of constant rank, then $\dom L$ is a coisotropic submanifold of $P$.

Here is a basic fact:

\begin{prop}
If $L: X \to Y$ and $L': Y \to Z$ are canonical relations with $(L,L')$ strongly transversal, then $L' \circ L: X \to Z$ is a canonical relation.
\end{prop}

In other words, the only obstacle to the composition of canonical relations being well-defined comes from smoothness concerns and not from the requirement that the resulting submanifold be Lagrangian.

\begin{rmk}
The composition of canonical relations is well-defined under weaker assumptions than strong transversality; in particular, it is well-defined under a \emph{clean intersection} hypothesis. We will not need this general notion.
\end{rmk}

\begin{defn}
A canonical relation $L: X \to Y$ is said to be a \emph{reduction} if, as a smooth relation, it is a surmersion; it is a \emph{coreduction} if it is a cosurmersion. We use $L: X \red Y$ to denote that $L$ is a reduction, and $L: X \cored Y$ to denote that $L$ is a coreduction.
\end{defn}

\begin{rmk}
A detailed study of some of the notions introduced above can be found in~\cite{LW} and~\cite{W3}
\end{rmk}

The use of the term ``reduction'' is motivated by the following example.

\begin{ex}(Symplectic Reduction)
Let $(M,\omega)$ be a symplectic manifold and $C$ a coisotropic submanifold. The distribution on $C$ given by $\ker \omega|_C \subset TC$, called the \emph{characteristic distribution} of $C$, is integrable and the induced foliation $C^\perp$ on $C$ is called the \emph{characteristic foliation} of $C$. If the leaf space $C/C^\perp$ has a smooth structure for which the projection $C \to C/C^\perp$ is a submersion, then $C/C^\perp$ naturally carries a symplectic structure and the relation
\[
red: M \red C/C^\perp
\]
assigning to an element of $C$ the leaf which contains it is a canonical relation which is a reduction in the sense above. The construction of $C/C^\perp$ from $M$ and $C$ is called \emph{symplectic reduction}. Symplectic reduction via Hamiltonian actions of Lie groups is a special case.
\end{ex}

\begin{ex}\label{ex:leaf}
We also note two more well-known examples of symplectic reduction. Suppose that $X$ is a manifold with $Y \subseteq X$ a submanifold. Then the restricted cotangent bundle $T^*X|_Y$ is a coisotropic submanifold of $T^*X$ whose reduction is symplectomorphic to $T^*Y$. Thus we obtain a reduction $T^*X \twoheadrightarrow T^*Y$.

Suppose now that $\F$ is a regular foliation on $X$ with smooth, Hausdorff leaf space $X/\F$. Then the conormal bundle $N^*\F$ is a coisotropic submanifold of $T^*X$ (this is in fact equivalent to the distribution $T\F$ being integrable) and its reduction is canonically symplectomorphic to $T^*(X/\F)$, giving rise to a reduction relation $T^*X \twoheadrightarrow T^*(X/\F)$.
\end{ex}

\subsection{The Symplectic Category}
We are now ready to introduce the category we will be working in, which we call a category even though, as previously mentioned, compositions are not always defined.

\begin{defn}
The \emph{symplectic category} is the category $\Symp$ whose objects are symplectic manifolds and whose morphisms are canonical relations.
\end{defn}

The category $\Symp$ thus defined has additional rich structure. In particular, it is a \emph{monoidal category}, where the tensor operation is given by Cartesian product and the unit is given by the symplectic manifold consisting of a single point $pt$. Moreover, $\Symp$ is \emph{symmetric monoidal} and \emph{rigid}, where the dualizing operation is given by $X \mapsto \overline X$ on objects and $L \mapsto L^t$ on morphisms. In addition, if we allow the empty set as a symplectic manifold, then it is simple to check that $\emptyset$ is both an initial and terminal object in this category, and that the categorical product of symplectic manifolds $X$ and $Y$ is the disjoint union $X \sqcup Y$.

\begin{rmk}
The same construction makes sense for general smooth relations between smooth manifolds; we denote the resulting category by $\SRel$ and call it the \emph{category of smooth relations}.
\end{rmk}

\subsection{The Cotangent Functor}

We define a functor $T^*: \Man \to \Symp$, called the \emph{cotangent functor}, as follows. First, $T^*$ assigns to a smooth manifold its cotangent bundle. To a smooth map $f: M \to N$, $T^*$ assigns the canonical relation $T^*f: T^*M \to T^*N$ given by
\[
T^*f: (p,(df)_p^*\xi) \mapsto (f(p),\xi).
\]
This is nothing but the composition $T^*M \to \overline{T^*M} \to T^*N$ of the Schwartz transform of $T^*M$ followed by the canonical relation given by the conormal bundle to the graph of $f$ in $M \times N$. We call $T^*f$ the \emph{cotangent lift} of $f$. It is a simple check to see that pairs of cotangent lifts are always strongly transversal and that $T^*$ really is then a functor: i.e. $T^*(f \circ g) = T^*f \circ T^*g$ and $T^*(id) = id$. Note that the same construction makes sense even when $f$ is only a smooth relation, giving a functor $T^*: \SRel \to \Symp$. 

\begin{ex}
When $\phi: M \to N$ is a diffeomorphism, $T^*\phi: T^*M \to T^*N$ is the graph of the usual lifted symplectomorphism.
\end{ex}

The following is easy to verify.

\begin{prop}
The cotangent lift of $f$ is a reduction if and only if $f$ is a surmersion; the cotangent lift of $g$ is a coreduction if and only if $g$ is cosurmersion.
\end{prop}

\subsection{Symplectic Monoids and Groupoids}
Since $\Symp$ is monoidal, we can speak about \emph{monoid objects} in $\Symp$:

\begin{defn}
A \emph{symplectic monoid} is a monoid object in $\Symp$. Thus, a symplectic monoid is a triple $(S,m,e)$ consisting of a symplectic manifold $S$ together with canonical relations
\[
m: S \times S \to S \text{ and } e: pt \to S,
\]
called the \emph{product} and \emph{unit} respectively, so that
\begin{center}
\begin{tikzpicture}[>=angle 90]
	\node (UL) at (0,1) {$S \times S \times S$};
	\node (UR) at (3,1) {$S \times S$};
	\node (LL) at (0,-1) {$S \times S$};
	\node (LR) at (3,-1) {$S$};
	
	\tikzset{font=\scriptsize};
	\draw[->] (UL) to node [above] {$id \times m$} (UR);
	\draw[->] (UL) to node [left] {$m \times id$} (LL);
	\draw[->] (UR) to node [right] {$m$} (LR);
	\draw[->] (LL) to node [above] {$m$} (LR);
\end{tikzpicture}
\end{center}
and
\begin{center}
\begin{tikzpicture}[thick]
	\node (UL) at (0,1) {$S$};
	\node (UR) at (3,1) {$S \times S$};
	\node (URR) at (6,1) {$S$};
	\node (LR) at (3,-1) {$S$};

	\tikzset{font=\scriptsize};	
	\draw[->] (UL) to node [above] {$e \times id$} (UR);
	\draw[->] (UR) to node [right] {$m$} (LR);
	\draw[->] (URR) to node [above] {$id \times e$} (UR);
	\draw[->] (UL) to node [above] {$id$} (LR);
	\draw[->] (URR) to node [above] {$id$} (LR);
\end{tikzpicture}
\end{center}
commute. We also require that all compositions involved be strongly transversal. We often refer to $S$ itself as a symplectic monoid, and use subscripts in the notation for the structure morphisms if we need to be explicit.
\end{defn}

\begin{ex}
Recall that a symplectic groupoid $S \rightrightarrows P$ is a Lie groupoid where $S$ is equipped with a symplectic form $\omega$ such that $m^*\omega = pr_1^*\omega+pr_2^*\omega$, where $pr_1,pr_2: S \times_P S \to S$ are the two projections and $m$ is the groupoid multiplication. This requirement is equivalent to the claim that the graph of $m$ in $\overline S \times \overline S \times S$ be Lagrangian. The image of the unit embedding $P \to S$ of any symplectic groupoid is a Lagrangian submanifold of $S$, so $S$ together with $m$ and the canonical relation $pt \to S$ given by this image is a symplectic monoid.
\end{ex}

Zakrzewski gave in \cite{SZ1}, \cite{SZ2} a complete characterization of symplectic groupoids in terms of such structures, or more specifically, symplectic monoids equipped with a $*$-structure:

\begin{defn}
A \emph{*-structure} on a symplectic monoid $S$ is an anti-symplectomorphism $s: S \to S$ (equivalently a symplectomorphism $s: \overline S \to S$) such that $s^2 = id$ and the diagram
\begin{center}
\begin{tikzpicture}[>=angle 90]
	\node (UL) at (0,1) {$\overline{S} \times \overline{S}$};
	\node (U) at (3,1) {$\overline{S} \times \overline{S}$};
	\node (UR) at (6,1) {$S \times S$};
	\node (LL) at (0,-1) {$\overline{S}$};
	\node (LR) at (6,-1) {$S,$};

	\tikzset{font=\scriptsize};
	\draw[->] (UL) to node [above] {$\sigma$} (U);
	\draw[->] (U) to node [above] {$s \times s$} (UR);
	\draw[->] (UL) to node [left] {$m$} (LL);
	\draw[->] (UR) to node [right] {$m$} (LR);
	\draw[->] (LL) to node [above] {$s$} (LR);
\end{tikzpicture}
\end{center}
where $\sigma$ is the symplectomorphism exchanging components, commutes. A symplectic monoid equipped with a $*$-structure will be called a \emph{symplectic $*$-monoid}.

A $*$-structure $s$ is said to be \emph{strongly positive} if the diagram
\begin{center}
\begin{tikzpicture}[>=angle 90]
	\node (UL) at (0,1) {$S \times \overline{S}$};
	\node (UR) at (3,1) {$S \times S$};
	\node (LL) at (0,-1) {$pt$};
	\node (LR) at (3,-1) {$S,$};

	\tikzset{font=\scriptsize};
	\draw[->] (UL) to node [above] {$id \times s$} (UR);
	\draw[->] (LL) to node [left] {$$} (UL);
	\draw[->] (UR) to node [right] {$m$} (LR);
	\draw[->] (LL) to node [above] {$e$} (LR);
\end{tikzpicture}
\end{center}
where $pt \to S \times \overline{S}$ is the morphism given by the diagonal of $S \times \overline{S}$, commutes.
\end{defn}

\begin{thrm}[Zakrzewski, \cite{SZ1}\cite{SZ2}]
Symplectic groupoids are in $1$-$1$ correspondence with strongly positive symplectic $*$-monoids.
\end{thrm}

The $*$-structure $s$ in Zakrzewski's correspondence serves as the inverse of the symplectic groupoid structure on $S \rightrightarrows P$, where $P \subseteq S$ is the Lagrangian submanifold defining the unit relation $e: pt \to S$. The source and target of $S \rightrightarrows P$ can be extracted from the requirement that $*$ be strongly positive.

\begin{rmk}
There is a similar characterization of Lie groupoids as strongly positive $*$-monoids in the category of smooth relations.
\end{rmk}

Reversing the arrows in the definition above leads to the notion of a \emph{symplectic comonoid}; we will call the structure morphisms of a symplectic comonoid the \emph{coproduct} and \emph{counit}, and will denote them by $\Delta$ and $\varepsilon$ respectively. Similarly, one can speak of a (strongly positive) $*$-structure on a symplectic comonoid.

\begin{ex}\label{ex:std-com}
Let $M$ be a manifold. Then $T^*M$ has a natural symplectic $*$-comonoid structure, obtained by reversing the arrows in its standard symplectic groupoid structure. To be explicit, the coproduct $T^*M \to T^*M \times T^*M$ is
\[
\Delta: (p,\xi+\eta) \mapsto ((p,\xi),(p,\eta)),
\]
which is obtained as the cotangent lift of the standard diagonal map $M \to M \times M$, and the counit $\varepsilon: T^*M \to pt$ is given by the zero section and is obtained as the cotangent lift of the canonical map $M \to pt$. The $*$-structure is the Schwartz transform.
\end{ex}

There is also a natural notion of a morphism between symplectic monoids:
\begin{defn}
Suppose $(S,m,e)$ and $(S',m',e')$ are symplectic monoids. A \emph{monoid morphism} $S \to S'$ is a canonical relation $L: S \to S'$ for which the following diagrams commute:
\begin{center}
\begin{tikzpicture}[>=angle 90]
	\node (UL) at (0,.5) {$S \times S$};
	\node (UR) at (3,.5) {$S' \times S'$};
	\node (LL) at (0,-1) {$S$};
	\node (LR) at (3,-1) {$S'$};
	
	\node (M) at (4.75,0) {$\text{and}$};
	
    \node (U2) at (7,.5) {$pt$};
	\node (LL2) at (6,-1) {$S$};
	\node (LR2) at (8,-1) {$S'$};

	\tikzset{font=\scriptsize};
	\draw[->] (UL) to node [above] {$L \times L$} (UR);
	\draw[->] (UL) to node [left] {$m$} (LL);
	\draw[->] (UR) to node [right] {$m'$} (LR);
	\draw[->] (LL) to node [above] {$L$} (LR);
	
	\draw[->] (U2) to node [left] {$e$} (LL2);
	\draw[->] (U2) to node [right] {$e'$} (LR2);
	\draw[->] (LL2) to node [above] {$L$} (LR2);
\end{tikzpicture}
\end{center}
and all compositions are strongly transversal. Reversing all arrows gives the notion of a \emph{comonoid morphism} between symplectic comonoids, and there is an obvious way to phrase the compatibility between a (co)monoid morphism and a given $*$-structure as well.
\end{defn}

\subsection{Hopf Algebra Objects}
\begin{defn}
A \emph{Hopf algebra object} in $\Symp$ consists of a symplectic manifold $S$ together with
\begin{itemize}
\item a symplectic monoid structure $(S,m,e)$,
\item a symplectic comonoid structure $(S,\Delta,\varepsilon)$, and
\item a symplectomorphism $i: S \to S$
\end{itemize}
such that the following diagrams commute, with all compositions strongly transversal:
\begin{itemize}
\item (compatibility between product and coproduct)
\begin{center}
\begin{tikzpicture}[>=angle 90]
	\node (U1) at (0,1) {$S \times S$};
	\node (U2) at (3,1) {$S$};
	\node (U3) at (6,1) {$S \times S$};
	\node (L1) at (0,-1) {$S \times S \times S \times S$};
	\node (L3) at (6,-1) {$S \times S \times S \times S$};

	\tikzset{font=\scriptsize};
	\draw[->] (U1) to node [above] {$m$} (U2);
	\draw[->] (U2) to node [above] {$\Delta$} (U3);
	\draw[->] (U1) to node [left] {$\Delta \times \Delta$} (L1);
	\draw[->] (L3) to node [right] {$m \times m$} (U3);
	\draw[->] (L1) to node [above] {$id \times \sigma \times id$} (L3);
\end{tikzpicture}
\end{center}
where $\sigma: S \times S \to S \times S$ is the symplectomorphism exchanging components,
\item (compatibilities between product and counit, between coproduct and unit, and between unit and counit respectively)
\begin{center}
\begin{tikzpicture}[>=angle 90]
	\node (U1) at (0,1) {$S \times S$};
	\node (U2) at (3,1) {$S$};
	\node (L1) at (1.5,-1) {$pt$};
	
	\node at (3,0) {$\text{,}$};
	
	\node (U3) at (4,1) {$S$};
	\node (U4) at (7,1) {$S \times S$};
	\node (L2) at (5.5,-1) {$pt$};
	
	\node at (8,0) {, \quad and };
	
	\node (U5) at (10.5,1) {$S$};
	\node (L5) at (9,-1) {$pt$};
	\node (L6) at (12,-1) {$pt$};

	\tikzset{font=\scriptsize};
	\draw[->] (U1) to node [above] {$m$} (U2);
	\draw[->] (U1) to node [left] {$\varepsilon \times \varepsilon$} (L1);
	\draw[->] (U2) to node [right] {$\varepsilon$} (L1);
	
	\draw[->] (U3) to node [above] {$\Delta$} (U4);
	\draw[->] (L2) to node [left] {$e$} (U3);
	\draw[->] (L2) to node [right] {$e \times e$} (U4);
	
	\draw[->] (L5) to node [left] {$e$} (U5);
	\draw[->] (U5) to node [right] {$\varepsilon$} (L6);
	\draw[->] (L5) to node [above] {$id$} (L6);
\end{tikzpicture}
\end{center}
\item (antipode conditions)
\begin{center}
\begin{tikzpicture}[>=angle 90]
	\node (U1) at (0,1) {$S \times S$};
	\node (U3) at (4,1) {$S \times S$};
	\node (L1) at (0,-1) {$S$};
	\node (L2) at (2,-1) {$pt$};
	\node (L3) at (4,-1) {$S$};
	
	\node at (5.5,0) {$\text{and}$};
	
	\node (U4) at (7,1) {$S \times S$};
	\node (U6) at (11,1) {$S \times S$};
	\node (L4) at (7,-1) {$S$};
	\node (L5) at (9,-1) {$pt$};
	\node (L6) at (11,-1) {$S$};

	\tikzset{font=\scriptsize};
	\draw[->] (U1) to node [above] {$id \times i$} (U3);
	\draw[->] (L1) to node [left] {$\Delta$} (U1);
	\draw[->] (U3) to node [right] {$m$} (L3);
	\draw[->] (L1) to node [above] {$\varepsilon$} (L2);
	\draw[->] (L2) to node [above] {$e$} (L3);
	
	\draw[->] (U4) to node [above] {$i \times id$} (U6);
	\draw[->] (L4) to node [left] {$\Delta$} (U4);
	\draw[->] (U6) to node [right] {$m$} (L6);
	\draw[->] (L4) to node [above] {$\varepsilon$} (L5);
	\draw[->] (L5) to node [above] {$e$} (L6);
\end{tikzpicture}
\end{center}
\end{itemize}
\end{defn}

\begin{ex}\label{ex:hopf}
Equip the cotangent bundle $T^*G$ of a Lie group $G$ with the comonoid structure of Example~\ref{ex:std-com} and the monoid structure coming from its symplectic groupoid structure over $\g^*$. Then these two structures together with the ``antipode'' $T^*i$, where $i: G \to G$ is inversion, make $T^*G$ into a Hopf algebra object in the symplectic category. 

Note that we have the same structure on the cotangent bundle of a more general Lie groupoid, but this will not form a Hopf algebra object; in particular, the antipode conditions fail owing to the fact that the groupoid product is not defined on all of $G \times G$. Later we will see that this can be fixed by introducing a ``base'' more general than $pt$; this provides explicit examples of ``hopfoids'' satisfying our Definition~\ref{defn:symp-hopf} which do not satisfy Zakrzewski's definitions.
\end{ex}

\begin{rmk}
An alternate notion of a groupoid-like object in the symplectic category is introduced in~\cite{CC}. This concept of a \emph{relational symplectic groupoid} is similar to that of a symplectic $*$-monoid except that the compatibilities between multiplication and the unit hold only up to an equivalence relation and not necessarily strictly. All of the symplectic $*$-monoids the constructions outlined in this paper produce can thus be viewed as examples of relational symplectic groupoids, but our notion of a symplectic hopfoid differs in that it brings in a general type of ``base'' with corresponding source and target morphisms. It would be interesting to explore further the relation between relational symplectic groupoids and symplectic hopfoids.
\end{rmk}

There is a natural notion of morphism between Hopf algebra objects in $\Symp$, which we briefly return to in our final discussion:
\begin{defn}\label{defn:hopf-morphism}
Suppose $S$ and $S'$ are Hopf algebra objects in $\Symp$. A \emph{morphism} $S \to S'$ is a canonical relation $L: S \to S'$ which: (i) is a monoid morphism for the symplectic monoid structures on $S$ and $S'$, (ii) is a comonoid morphism for the symplectic comonoid structures on $S$ and $S'$, and (iii) preserves the antipodes of $S$ and $S'$ in the obvious way.
\end{defn}

\section{Double Groupoids}\label{chap:dbl-grpds}

\subsection{Preliminaries}
\begin{defn}\label{defn:dbl-grpd}
A \emph{double Lie groupoid} is a diagram
\begin{center}
\begin{tikzpicture}[>=angle 90,scale=.85]
	\node at (0,0) {$H$};
	\node at (2,0) {$M$};
	\node at (0,2) {$D$};
	\node at (2,2) {$V$};

	\tikzset{font=\scriptsize};
	\draw[->] (.4,.08) -- (1.6,.08);
	\draw[->] (.4,-.08) -- (1.6,-.08);
	
	\draw[->] (-.08,1.6) -- (-.08,.4);
	\draw[->] (.08,1.6) -- (.08,.4);
	
	\draw[->] (.4,2+.08) -- (1.6,2+.08);
	\draw[->] (.4,2-.08) -- (1.6,2-.08);
	
	\draw[->] (2-.08,1.6) -- (2-.08,.4);
	\draw[->] (2+.08,1.6) -- (2+.08,.4);
\end{tikzpicture}
\end{center}
of Lie groupoids such that the structure maps of the top and bottom groupoids give homomorphisms from the left groupoid to the right groupoid, and vice-versa. We will refer to the  four groupoid structures involved as the top, bottom, left, and right groupoids. We also assume that the \emph{double source map}
\[
D \to H \times V
\]
is a surjective submersion. Also, we often refer to $D$ itself as the double Lie groupoid and to $M$ as its double base. When $M$ is a point, we will call $D$ a \emph{double Lie group}. Finally, we refer to the double groupoid obtained by exchanging the roles of $H$ and $V$ as the \emph{transposed} double groupoid.
\end{defn}

%
%
%
%

\begin{ex}\label{ex:dmain}
For any Lie groupoid $G \rightrightarrows M$, there is a double Lie groupoid structure on
\begin{center}
\begin{tikzpicture}[>=angle 90]
	\node (LL) at (0,0) {$M$};
	\node (LR) at (2,0) {$M.$};
	\node (UL) at (0,2) {$G$};
	\node (UR) at (2,2) {$G$};

	\tikzset{font=\scriptsize};
	\draw[->] (.4,.07) -- (1.6,.07);
	\draw[->] (.4,-.07) -- (1.6,-.07);
	
	\draw[->] (-.07,1.6) -- (-.07,.4);
	\draw[->] (.07,1.6) -- (.07,.4);
	
	\draw[->] (.4,2.07) -- (1.7,2.07);
	\draw[->] (.4,2-.07) -- (1.7,2-.07);
	
	\draw[->] (2-.07,1.6) -- (2-.07,.4);
	\draw[->] (2+.07,1.6) -- (2+.07,.4);
\end{tikzpicture}
\end{center}
Here, the left and right sides are the given groupoid structures, while the top and bottom are trivial groupoids.
\end{ex}

\begin{ex}\label{ex:dinertia}
Again for any Lie groupoid $G \rightrightarrows M$, there is a double Lie groupoid structure on
\begin{center}
\begin{tikzpicture}[>=angle 90]
	\node (LL) at (0,0) {$M \times M$};
	\node (LR) at (3,0) {$M.$};
	\node (UL) at (0,2) {$G \times G$};
	\node (UR) at (3,2) {$G$};

	\tikzset{font=\scriptsize};
	\draw[->] (.9,.07) -- (2.6,.07);
	\draw[->] (.9,-.07) -- (2.6,-.07);
	
	\draw[->] (-.07,1.6) -- (-.07,.4);
	\draw[->] (.07,1.6) -- (.07,.4);
	
	\draw[->] (.9,2.07) -- (2.5,2.07);
	\draw[->] (.9,2-.07) -- (2.5,2-.07);
	
	\draw[->] (3-.07,1.6) -- (3-.07,.4);
	\draw[->] (3.07,1.6) -- (3.07,.4);
\end{tikzpicture}
\end{center}
Here, the right side is the given groupoid structure, the top and bottom are pair groupoids, and the left is a product groupoid.
\end{ex}

For much of what follows, we will need a consistent labeling of the structure maps involved in the various groupoids considered. First, the source, target, unit, inverse, and product of the right groupoid $V \rightrightarrows M$ are respectively
\[
s_V(\cdot), t_V(\cdot), 1^V_\cdot, i_V(\cdot), m_V(\cdot,\cdot) \text{ or } \cdot \circ_V \cdot
\]
The structure maps of $H \rightrightarrows M$ will use the same symbols with $V$ replaced by $H$. Now, the structure maps of the top and left groupoids will use the same symbol as those of the \emph{opposite} structure with a tilde on top; so for example, the structure maps of $D \rightrightarrows H$ are
\[
\widetilde{s}_V(\cdot), \widetilde{t}_V(\cdot), \widetilde{1}^V_\cdot, \widetilde{i}_V(\cdot), \wt m_V(\cdot,\cdot) \text{ or } \cdot \widetilde{\circ}_V \cdot
\]

To emphasize: the structure maps of the groupoid structure on $D$ \emph{over} $V$ use an $H$, and those of the groupoid structure on $D$ \emph{over} $H$ use a $V$, or in other words the structure maps of the horizontal structures in \ref{defn:dbl-grpd} use $H$ while those of the vertical structures use $V$. This has a nice practical benefit in that it is simpler to keep track of the various relations these maps satisfy; for example, the maps $\widetilde{s}_{H}$ and $s_{H}$ give the groupoid homomorphism
\begin{center}
\begin{tikzpicture}[>=angle 90,scale=.75]
	\node (LL) at (0,0) {$H$};
	\node (LR) at (2,0) {$M,$};
	\node (UL) at (0,2) {$D$};
	\node (UR) at (2,2) {$V$};

	\tikzset{font=\scriptsize};
	\draw[->] (UL) to node [above] {$\widetilde{r}_{H}$} (UR);
	
	\draw[->] (-.07,1.6) -- (-.07,.4);
	\draw[->] (.07,1.6) -- (.07,.4);
	
	\draw[->] (2-.07,1.6) -- (2-.07,.4);
	\draw[->] (2+.07,1.6) -- (2+.07,.4);
	
	\draw[->] (LL) -- node [above] {$r_{H}$} (LR);
\end{tikzpicture}
\end{center}
so for instance we have: $t_V(\wt s_H(a)) = s_H(\wt t_V(a))$, $i_V(\wt s_H(a)) = \wt s_H(\wt i_V(a))$, $\wt s_H(\wt 1^V_v) = 1^V_{s_H(v)}$, etc. We will make extensive use of such identities.

We will denote elements of $D$ as squares with sides labeled by the possible sources and targets:
\begin{center}
\begin{tikzpicture}[scale=.75]
	\draw (0,0) rectangle (2,2);
	\node at (1,1) {$a$};

	\tikzset{font=\scriptsize};
	\node at (-.6,1) {$\wt t_V(a)$};
	\node at (1,2.4) {$\wt s_{H}(a)$};
	\node at (2.6,1) {$\wt s_V(a)$};
	\node at (1,-.4) {$\wt t_{H}(a)$};
\end{tikzpicture}
\end{center}
so that the left and right sides are the target and source of the left groupoid structure while the top and bottom sides are the source and target of the top groupoid structure. This lends itself well to compositions: two squares $a$ and $a'$ are composable in the left groupoid if the right side of the first is the left of the second, i.e. if $\widetilde{s}_{V}(a) = \widetilde{t}_{V}(a')$, and the composition $a \,\wt\circ_{V} a'$ in the left groupoid $D \rightrightarrows H$ can be viewed as ``horizontal concatenation'':
\begin{center}
\begin{tikzpicture}[scale=.75]
	\node at (1,1) {$a$};
	\node at (3,1) {$a'$};
	\node at (6,1) {$=$};
	\node at (9,1) {$a\, \wt\circ_V a'$};

	\tikzset{font=\scriptsize};
	\draw (0,0) rectangle (2,2);
	\node at (-.6,1) {$\wt t_V(a)$};
	\node at (1,2.4) {$\wt s_{H}(a)$};
	\node at (1,-.4) {$\wt t_{H}(a)$};
	
	\draw (2,0) rectangle +(2,2);
	\node at (3,2.4) {$\wt s_{H}(a')$};
	\node at (4.7,1) {$\wt s_V(a')$};
	\node at (3,-.4) {$\wt t_{H}(a')$};
	
	\draw (8,0) rectangle +(2,2);
	\node at (7.4,1) {$\wt t_V(a)$};
	\node at (9,2.4) {$\wt s_{H}(a) \circ_V \wt s_{H}(a')$};
	\node at (10.7,1) {$\wt s_V(a')$};
	\node at (9,-.4) {$\wt t_{H}(a) \circ_V \wt t_{H}(a')$};
\end{tikzpicture}
\end{center}
Similarly, the groupoid composition in the top groupoid $D \rightrightarrows V$ can be viewed as ``vertical concatenation''.

With these notations, the compatibility between the two groupoid products on $D$ can then be expressed as saying that composing vertically and then horizontally in
\begin{center}
\begin{tikzpicture}
	\draw (0,0) rectangle (1.5,1.5);
	\node at (.75,.75) {$a$};
	\draw (1.5,0) rectangle (3,1.5);
	\node at (2.25,.75) {$b$};
	\draw (0,1.5) rectangle (1.5,3);
	\node at (.75,2.25) {$c$};
	\draw (1.5,1.5) rectangle (3,3);
	\node at (2.25,2.25) {$d$};
\end{tikzpicture}
\end{center}
produces the same result as composing horizontally and then vertically, whenever all compositions involved are defined.

\begin{defn}
The \emph{core} of a double Lie groupoid $D$ is the submanifold $C$ of elements of $D$ whose sources are both units for the right $V \rightrightarrows M$ and bottom $H \rightrightarrows M$ groupoids; that is, the set of elements of the form
\begin{center}
\begin{tikzpicture}[scale=.75]
	\node at (1,1) {$a$};
	
	\tikzset{font=\scriptsize};
	\draw (0,0) rectangle (2,2);
	\node at (-.7,1) {$\wt t_{V}(a)$};
	\node at (1,2.4) {$1^V_m$};
	\node at (2.6,1) {$1^{H}_m$};
	\node at (1,-.4) {$\wt t_H(a)$};
\end{tikzpicture}
\end{center}
The condition on the double source map in the definition of a double Lie groupoid ensures that $C$ is a submanifold of $D$.
\end{defn}

\begin{thrm}[Brown-Mackenzie, \cite{BM}]
The core of a double Lie groupoid $D$ has a natural Lie groupoid structure over the double base $M$.
\end{thrm}

The groupoid structure on the core comes from a combination of the two groupoid structures on $D$. Explicitly, the groupoid product of two elements $c$ and $c'$ in the core can be expressed as composing vertically and then horizontally (or equivalently  horizontally and then vertically) in the following diagram:
\begin{center}
\begin{tikzpicture}
	\draw (0,0) rectangle (1.5,1.5);
	\node at (.75,.75) {$c$};
	\draw (1.5,0) rectangle (3,1.5);
	\node at (2.25,.75) {$\wt 1^{H}_{\wt t_{H}(c')}$};
	\draw (0,1.5) rectangle (1.5,3);
	\node at (.75,2.25) {$\wt 1^V_{\wt t_V(c')}$};
	\draw (1.5,1.5) rectangle (3,3);
	\node at (2.25,2.25) {$c'$};
\end{tikzpicture}
\end{center}

The core groupoid of Example~\ref{ex:dmain} is the trivial groupoid $M \rightrightarrows M$ while that of Example~\ref{ex:dinertia} is $G \rightrightarrows M$ itself.

\subsection{Symplectic Double Groupoids}
\begin{defn}
A \emph{symplectic double groupoid} is a double Lie groupoid $D$ where $D$ is equipped with a symplectic structure making the top and left groupoid structures in~(\ref{defn:dbl-grpd}) symplectic groupoids.
\end{defn}

It is well-known that a symplectic groupoid $S \rightrightarrows P$ induces a unique Poisson structure on $P$ with respect to which the source is a Poisson map. In the double groupoid setting, when $V$ and $H$ are equipped with these Poisson structures, a result of Mackenzie (\cite{M}) shows that the groupoids $V \rightrightarrows M$ and $H \rightrightarrows M$ are actually Poisson groupoids in duality, meaning that their Lie algebroids are dual to one another.

The double symplectic structure on $D$ endows the core with additional structure as follows:

\begin{thrm}[Mackenzie \cite{M}]
The core $C$ of a symplectic double groupoid $D$ is a symplectic submanifold of $D$ and the induced groupoid structure on $C \rightrightarrows M$ is that of a symplectic groupoid.
\end{thrm}

\begin{ex}\label{ex:main}
For any groupoid $G \rightrightarrows M$, there is a symplectic double groupoid structure on $T^*G$ of the form
\begin{center}
\begin{tikzpicture}[>=angle 90]
	\node (LL) at (0,0) {$A^*$};
	\node (LR) at (2,0) {$M$};
	\node (UL) at (0,2) {$T^*G$};
	\node (UR) at (2,2) {$G$};

	\tikzset{font=\scriptsize};
	\draw[->] (.4,.07) -- (1.6,.07);
	\draw[->] (.4,-.07) -- (1.6,-.07);
	
	\draw[->] (-.07,1.6) -- (-.07,.4);
	\draw[->] (.07,1.6) -- (.07,.4);
	
	\draw[->] (.6,2.07) -- (1.7,2.07);
	\draw[->] (.6,2-.07) -- (1.7,2-.07);
	
	\draw[->] (2-.07,1.6) -- (2-.07,.4);
	\draw[->] (2+.07,1.6) -- (2+.07,.4);
\end{tikzpicture}
\end{center}
Here, $A$ is the Lie algebroid of $G \rightrightarrows M$, the right groupoid structure is the given one on $G$, the top and bottom are the natural groupoid structures on vector bundles given by fiber-wise addition, and the left structure is the induced symplectic groupoid structure on the cotangent bundle of a Lie groupoid. The core of this symplectic double groupoid is symplectomorphic to $T^*M$, and the core groupoid is simply $T^*M \rightrightarrows M$. Note, in particular, that when $G$ is a Lie group $T^*G$ is a symplectic double group.
\end{ex}

\begin{ex}\label{ex:inertia}
Again for any groupoid $G \rightrightarrows M$, there is a symplectic double groupoid structure on $\overline{T^*G} \times T^*G$ of the form
\begin{center}
\begin{tikzpicture}[>=angle 90]
	\node (LL) at (0,0) {$\overline{A^*} \times A^*$};
	\node (LR) at (3,0) {$A^*$};
	\node (UL) at (0,2) {$\overline{T^*G} \times T^*G$};
	\node (UR) at (3,2) {$T^*G$};

	\tikzset{font=\scriptsize};
	\draw[->] (.9,.07) -- (2.6,.07);
	\draw[->] (.9,-.07) -- (2.6,-.07);
	
	\draw[->] (-.07,1.6) -- (-.07,.4);
	\draw[->] (.07,1.6) -- (.07,.4);
	
	\draw[->] (1.2,2.07) -- (2.5,2.07);
	\draw[->] (1.2,2-.07) -- (2.5,2-.07);
	
	\draw[->] (3-.07,1.6) -- (3-.07,.4);
	\draw[->] (3.07,1.6) -- (3.07,.4);
\end{tikzpicture}
\end{center}
Here, the right side is the induced symplectic groupoid structure on $T^*G$, the top and bottom are pair groupoids, and the left is a product groupoid. The core is symplectomorphic to $T^*G$, and the core groupoid is $T^*G \rightrightarrows A^*$.
\end{ex}

Both of the above examples are special cases of the following result due to Mackenzie:

\begin{thrm}[Mackenzie \cite{M}]\label{thrm:ctdbl}
Let $D$ be a double Lie groupoid. Then the cotangent bundle $T^*D$ has a natural symplectic double groupoid structure
\begin{center}
\begin{tikzpicture}[>=angle 90]
	\node (LL) at (0,0) {$A^*H$};
	\node (LR) at (3,0) {$A^*C$};
	\node (UL) at (0,2) {$T^*D$};
	\node (UR) at (3,2) {$A^*V$};

	\tikzset{font=\scriptsize};
	\draw[->] (.5,.07) -- (2.5,.07);
	\draw[->] (.5,-.07) -- (2.5,-.07);
	
	\draw[->] (-.07,1.6) -- (-.07,.4);
	\draw[->] (.07,1.6) -- (.07,.4);
	
	\draw[->] (.6,2.07) -- (2.5,2.07);
	\draw[->] (.6,2-.07) -- (2.5,2-.07);
	
	\draw[->] (3-.07,1.6) -- (3-.07,.4);
	\draw[->] (3+.07,1.6) -- (3+.07,.4);
\end{tikzpicture}
\end{center}
where $A^*H$ and $A^*V$ are the duals of the Lie algebroids of $H \rightrightarrows M$ and $V \rightrightarrows M$ respectively, and $A^*C$ is the dual of the Lie algebroid of the core groupoid $C \rightrightarrows M$. The core of this symplectic double groupoid is symplectomorphic to $T^*C$ where $C$ is the core of $D$.
\end{thrm}

In this paper, in addition to the descriptions of the top and left cotangent groupoids, we will only need to use the units of the right and bottom groupoids: the unit of $A^*V \rightrightarrows A^*C$ is  induced by that of $T^*D \rightrightarrows A^*H$, and so comes from the identification of $A^*H$ with the conormal bundle $N^*H \subset T^*D$, and similarly the unit of $A^*H \rightrightarrows A^*C$ comes from the identification of $A^*V$ with $N^*V$. Example \ref{ex:main} arises from applying this theorem to the double groupoid
\begin{center}
\begin{tikzpicture}[>=angle 90]
	\node (LL) at (0,0) {$M$};
	\node (LR) at (2,0) {$M,$};
	\node (UL) at (0,2) {$G$};
	\node (UR) at (2,2) {$G$};

	\tikzset{font=\scriptsize};
	\draw[->] (.4,.07) -- (1.6,.07);
	\draw[->] (.4,-.07) -- (1.6,-.07);
	
	\draw[->] (-.07,1.6) -- (-.07,.4);
	\draw[->] (.07,1.6) -- (.07,.4);
	
	\draw[->] (.4,2.07) -- (1.7,2.07);
	\draw[->] (.4,2-.07) -- (1.7,2-.07);
	
	\draw[->] (2-.07,1.6) -- (2-.07,.4);
	\draw[->] (2+.07,1.6) -- (2+.07,.4);
\end{tikzpicture}
\end{center}
and Example \ref{ex:inertia} arises from the double groupoid
\begin{center}
\begin{tikzpicture}[>=angle 90]
	\node (LL) at (0,0) {$M \times M$};
	\node (LR) at (3,0) {$M.$};
	\node (UL) at (0,2) {$G \times G$};
	\node (UR) at (3,2) {$G$};

	\tikzset{font=\scriptsize};
	\draw[->] (.9,.07) -- (2.6,.07);
	\draw[->] (.9,-.07) -- (2.6,-.07);
	
	\draw[->] (-.07,1.6) -- (-.07,.4);
	\draw[->] (.07,1.6) -- (.07,.4);
	
	\draw[->] (.9,2.07) -- (2.5,2.07);
	\draw[->] (.9,2-.07) -- (2.5,2-.07);
	
	\draw[->] (3-.07,1.6) -- (3-.07,.4);
	\draw[->] (3.07,1.6) -- (3.07,.4);
\end{tikzpicture}
\end{center}

\begin{rmk}
Given a symplectic double groupoid $D$, the two groupoid structures on $D$ can be thought of as giving compatible monoid and comonoid structures in the symplectic category, resulting in a type of ``bialgebra''-object. However, as alluded to previously in Example~\ref{ex:hopf}, using the two groupoid inverses on $D$ to produce a type of antipode does \emph{not} in fact produce a Hopf algebra object in the symplectic category; obtaining such an object requires the use of a ``base'' more general than a point, which is what leads to our notion of a symplectic hopfoid where the correct ``base'' is given by the core.
\end{rmk}

\subsection{Realizing the Core via Reduction}\label{sec:reduc}
We now describe a procedure for producing the core of a symplectic double groupoid, and indeed the symplectic groupoid structure on the core, via symplectic reduction. For the remainder of what follows, we assume that the target and source fibers of the groupoids we consider are all connected.

Let $D$ be a symplectic double groupoid. The unit submanifold $1^V M$ of the groupoid structure on $V$ is coisotropic in $V$ since $V \rightrightarrows M$ is a Poisson groupoid. Hence its preimage
\[
X := \widetilde{s}_{H}^{-1}(1^{V} M) \subseteq D
\]
under the source of the top groupoid (which is a Poisson map) is a coisotropic submanifold of $D$. In square notation $X$ consists of those elements of the form:
\begin{center}
\begin{tikzpicture}[scale=.75]
	\node at (1,1) {$a$};

	\tikzset{font=\scriptsize};
	\draw (0,0) rectangle (2,2);
	\node at (-.6,1) {$\wt t_V(a)$};
	\node at (1,2.4) {$1^V_m$};
	\node at (2.6,1) {$\wt s_V(a)$};
	\node at (1,-.4) {$\wt t_{H}(a)$};
\end{tikzpicture}
\end{center}
Note that the core $C$ of $D$ sits inside of $X$. Similarly, the same construction using the bottom and left groupoids in~\ref{defn:dbl-grpd} produces the coisotropic submanifold
\[
Y := \wt s_V^{-1}(1^{H}M) \subseteq D
\]
of $D$, which also contains the core.

\begin{ex}
Consider the symplectic double groupoid of Example \autoref{ex:main}. The submanifold $X$ in this case is the restricted cotangent bundle $T^*G|_M \subseteq T^*G$. As described in Example~\ref{ex:leaf}, the reduction of this coisotropic submanifold by its characteristic foliation is the core $T^*M$, and the resulting reduction relation
\[
T^*G \red T^*M
\] 
is the transpose of the cotangent lift $T^*e$ of the unit embedding $e: M \to G$ of the groupoid $G \rightrightarrows M$.

Now, performing this procedure using the left groupoid is more interesting. Recall that the source map $\wt s$  ($\wt s_V$ in the notation we have been using for general symplectic double groupoids) of the groupoid $T^*G \rightrightarrows A^*$ is determined by the requirement that
\[
\wt s(g,\xi)\big|_{\ker (dt)_{e(s(g))}} = (dL_g)_{e(s(g))}^*\left(\xi\big|_{\ker (dt)_g}\right)
\]
where $L_g: t^{-1}(s(g)) \to t^{-1}(t(g))$ is the left groupoid multiplication by $g$ and we identify $A^* \cong N^*M$. Here $t$ and $s$ are the target and source maps of $G \rightrightarrows M$ respectively. Thus, we see that $(g,\xi)$ maps to a unit $(s(g),0) \in A^*$ of the bottom groupoid (note that the unit of the bottom groupoid is given by the zero section $0_M$ of $A^*$) if and only if $\xi|_{\ker (dt)_g} = 0$, and hence equivalently if and only if $\xi$ is in the image of $(dt)_g^*$. 

Therefore, the coisotropic submanifold $Y := \wt s^{-1}(0_M)$ of $T^*G$ is isomorphic to $N^*\F_t$, where $\F_t$ is the foliation of $G$ given by the $t$-fibers of the groupoid $G \rightrightarrows M$. According to Example~\ref{ex:leaf}, the reduction of this is also the core $T^*M$, and the reduction relation $T^*G \red T^*M$ turns out to be the cotangent lift $T^*t$. 
\end{ex}

These results generalize in the following way to an arbitrary symplectic double groupoid. We first have the following explicit description of the characteristic foliation $X^\perp$ of $X := \widetilde{s}_{H}^{-1}(1^V M)$:

\begin{lem}\label{lem:fol}
The leaf of the foliation $X^\perp$ containing $a \in X$ is given by
\begin{equation}\label{char-fol}
X^\perp_a := \{ a\ \wt\circ_{H}\,\wt 1^V_\lambda\ |\ \lambda \in t_{H}^{-1}(m)\},
\end{equation}
where $m = s_V(\wt s_{H}(a))$.
\end{lem}

\begin{proof}
To show that the characteristic foliation is as claimed, we must show that
\[
T_a (X^\perp_a) = (T_a X)^\perp
\]
where $(T_a X)^\perp$ is the symplectic orthogonal of $T_a X$ in $T_aD$. A dimension count shows that these two spaces have the same dimension. Indeed, for $m = \dim M, v = \dim V$, and $h = \dim H$, using the fact that each $\widetilde s_H$-fiber has dimension $\dim D - v$ we have:
\[ \dim(T_a X) = \dim D-v+m, \text{ so } \dim(T_aX)^\perp = \dim D-(\dim D-v+m) = v-m, \]
and $\dim (X_a^\perp) = h-m$ since elements of $X_a^\perp$ are parametrized by $t_H^{-1}(m)$; since $V$ and $H$ are both Lagrangian submanifolds of $D$, $v=h$ and the claim follows. Thus we only need to show that the former space in the claimed equality above is contained in the latter. Thus we must show that if $Y \in T_a(X^\perp_a)$, then
\[
\omega_a(Y,Z) = 0
\]
for any $Z \in T_a X$ where $\omega$ is the symplectic form on $D$. Since
\[
T_a X = (d\wt s_{H})_a^{-1}\left(T_{1^V_m}(1^V M)\right),
\]
this means that $Z$ satisfies $(d\wt s_{H})_a Z \in T_{1^V_m}(1^V M)$.

We use the following explicit description of $T_a(X^\perp_a)$. The elements of $X_a^\perp$ are parametrized by $t_{H}^{-1}(m)$, so we have that $X^\perp_a$ is the image of the map
\[
\ell_{H}^{-1}(m) \to D
\]
given by the composition
\[
\lambda \mapsto \wt 1^V_\lambda \mapsto \wt L_a^{H} (\wt 1^V_\lambda)
\]
where $\wt L_a^{H}$ is left-multiplication by $a$ in the top groupoid. Taking differentials at $\lambda = 1^{H}_m$ then gives the explicit description of $T_a(X^\perp_a)$ we want; in particular, we can write $Y \in T_a(X^\perp_a)$ as
\begin{equation}\label{eq:tang-dist}
Y = (d \wt L_a^{H})_{\wt 1^V(1^{H}_m)}(d \wt 1^V)_{1^{H}_m} W
\end{equation}
for some $W \in \ker (dt_{H})_{1^{H}_m}$.

First we consider the case where $a$ is a unit of the top groupoid structure, so suppose that $a = \wt 1^{H}(1^V_m)$. In this case $\wt L_a^H$ is simply the identity, so the expression for $Y$ above becomes
\[
Y = (d \wt 1^V)_{1^{H}_m} W.
\]
Using the splitting
\[
TD|_V = TV \oplus \ker d\wt s_{H}|_V
\]
we can write $Z \in T_aX$ as
\[
Z = (d\wt 1^{H})_{1^V_m}(d\wt s_{H})_a Z + [Z- (d\wt 1^{H})_{1^V_m}(d\wt s_{H})_a Z],
\]
where the first term is tangent to the units of the top groupoid and the second term is in $\ker (d\wt s_{H})_a$, and so tangent to the $\wt s_H$-fiber through $a$. Then we have
\[
\omega_a (Y,Z) = \omega_a ((d \wt 1^V)_{1^{H}_m} W,(d\wt 1^{H})_{1^V_m}(d\wt s_{H})_a Z) +  \omega_a ((d \wt 1^V)_{1^{H}_m} W,[V - (d\wt 1^{H})_{1^V_m}(d\wt s_{H})_a Z]).
\]
Since $W \in \ker(dt_{H})_{1^{H}_m}$, one can check that $Y = (d \wt 1^V)_{1^{H}_m} W$ is tangent to the $\wt t_H$-fiber through $a$, so the second term above vanishes since source and target fibers of a symplectic groupoid are symplectically orthogonal to one another (see~\cite{CDW}). By the defining property of $Z$, we have $(d\wt s_{H})_a Z = (d 1^V)_m z$ for some $z \in T_m M$. Using this and the fact that $\wt 1^{H} \circ 1^V = \wt 1^V \circ 1^{H}$ (which follows from the double groupoid compatibilities), we can write the first term above as
\[
\omega_a ((d \wt 1^V)_{1^{H}_m} W,(d \wt 1^V)_{1^{H}_m}(d 1^{H})_m z),
\]
which vanishes since the embedding $\wt 1^V: H \to D$ is Lagrangian. Thus $\omega_a(Y,Z) = 0$ as was to be shown. 

Now, for the general case, given $a \in X$ set $p := \wt 1_H(1^V_m)$ and choose a local Lagrangian bisection $B$ containing $a$ of the top groupoid $D \rightrightarrows V$, which is known to always exist~\cite{CDW}. Recall that this is a Lagrangian submanifold of $D$ such that the restrictions
\[
\wt t_H|_B: B \to U \text{ and } \wt s_H|_B: B \to U'
\]
are diffeomorphisms of $B$ onto open subsets $U$ and $U'$ of $V$. This then defines a left-multiplication $L_B: \wt t_H^{-1}(U') \to \wt t_H^{-1}(U)$ between the open submanifolds $\wt t_H^{-1}(U')$ and $\wt t_H^{-1}(U)$ of $D$ by
\[
L_B(b) = (\wt s_H|_B)^{-1}(\wt t_H(b))\, \wt\circ_H\, b,
\]
which sends the leaf $X^\perp_p$ of~(\ref{char-fol}) to $X^\perp_a$. The map $L_B$ is actually a symplectomorphism and so sends the symplectic orthogonal $(T_p X)^\perp$ at $p$ to $(T_a X)^\perp$. Since $(T_{p} X)^\perp = T_{p}(X_{p}^\perp)$ by what we showed previously and $T_{p}(X_{p}^\perp)$ is sent to $T_a(X^\perp_a)$, we have our result.
\end{proof}

\begin{rmk}
Note that it is the entire double groupoid structure which makes this explicit description possible. As a contrast, given a symplectic realization $f: S \to P$ of a Poisson manifold $P$ and a coisotropic submanifold $N \subseteq P$, the characteristic foliation on $f^{-1}(N) \subseteq S$ is not easily described.
\end{rmk}

In square notation, the leaf of the characteristic foliation through $a$ consists of elements of the form
\begin{center}
\begin{tikzpicture}[scale=.85]
	\node at (1,1) {$a\,\,\wt\circ_{H}\wt 1^V_\lambda$};

	\tikzset{font=\scriptsize};
	\draw (0,0) rectangle (2,2);
	\node at (-1.2,1) {$\wt t_V(a) \circ_{H} \lambda$};
	\node at (1,2.4) {$1^V_{s_{H}(\lambda)}$};
	\node at (3.2,1) {$\wt s_V(a) \circ_{H} \lambda$};
	\node at (1,-.4) {$\wt t_{H}(a)$};
\end{tikzpicture}
\end{center}
where $\lambda \in t_{H}^{-1}(m)$. Switching the roles of the top and left groupoids in~\ref{defn:dbl-grpd} , we get that the leaves of the characteristic foliation of $Y := \wt s_V^{-1}(1^{H}M)$ are given by
\[
Y^\perp_a := \{a\ \wt\circ_V\,\wt 1^{H}_{\lambda}\ |\ \lambda \in t_V^{-1}(m)\},
\]
where $m = s_{H}(\wt s_V(a))$

Returning to the characteristic leaves of $X$, note that such a leaf intersects the core in exactly one point, since there is only one choice of $\lambda$ which will make $\wt s_V(a) \circ_{H} \lambda$ a unit for the left groupoid structure, namely $\lambda = i_{H}(\wt s_V(a))$. Thus, the core forms a cross section to the characteristic foliation of $X$ and we conclude that the leaf space $X/X^\perp$ of this characteristic foliation can be identified with the core. This leaf space is then naturally symplectic and we have:

\begin{prop}
The symplectic structure on the core obtained via the reduction above agrees with the one $C$ inherits as a symplectic submanifold of $D$.
\end{prop}

\begin{proof}
Let $i: C \to D$ be the inclusion of the core into $D$, and let $\pi: X \to C$ be the surjective submersion sending $a \in X$ to the characteristic leaf containing it, where we have identified the leaf space $X/X^\perp$ with $C$ in the above manner. Let $k: C \to X$ be the inclusion of the core into $X$; this is a section of $\pi$. Finally, let $j: X \to D$ be the inclusion of $X$ into $D$.

The symplectic form $\omega_C$ on $C$ obtained via reduction is characterized by the property that $j^*\omega = \pi^*\omega_C$. Since $i = j \circ k$, we have
\begin{align*}
i^*\omega &= (j \circ k)^*\omega \\
&= k^*(j^*\omega) \\
&= k^*(\pi^*\omega_C) \\
&= (\pi \circ k)^*\omega_C,
\end{align*}
which equals $\omega_C$ since $k$ is a section of $\pi$. This proves the claim.
\end{proof}

Explicitly, the reduction relation $D \red C$, which we will call $\mathbf{e}^t$ for reasons to be made clear later, obtained by reducing $X$ is given by
\[
\mathbf{e}^t: a \mapsto a\, \wt\circ_{H} \wt 1^V_{i_{H}(\wt s_V(a))} \text{ for } a \in \wt s_{H}^{-1}(1^V M) = X.
\]
For future reference, the transposed relation $\mathbf{e}: C \cored D$ (which is a coreduction) is given by
\begin{equation}\label{unit}
\mathbf{e}: a \mapsto a \wt\circ_{H} \wt 1^V_\lambda, \text{ for $\lambda \in H$ such that } t_{H}(\lambda) = s_V(\wt s_{H}(s)).
\end{equation}

Similarly, the reduction of the coisotropic $Y = \wt s_V^{-1}(1^{H}M) \subseteq D$ can also be identified with the core via similar maps, after simply exchanging the roles of $V$ and $H$. The reduction relation $D \red C$ obtained by reducing $Y$ will be called $\mathbf{t}$ and is explicitly given by
\begin{equation}\label{target}
\mathbf{t}: a \mapsto a\,\wt\circ_V\,\wt 1^{H}_{i_{V}(\wt s_{H}(a))} \text{ for } a \in \wt s_V^{-1}(1^{H} M) = Y.
\end{equation}

Similar results hold for the preimages of units under the target maps. To be clear, let $Z$ now be $\wt t_{H}^{-1}(1^V M)$, the preimage of the units of $V$ under the target map of the top groupoid. This is again coisotropic in $D$, and the leaf of the characteristic foliation through a point $a \in Z$ is now given by
\[
\{\wt 1^V_{\lambda}\,\wt\circ_{H}\,a\ |\ \lambda \in s_{H}^{-1}(m)\}
\]
where $m = t_V(\wt t_{H}(a))$. Similar to the above, we can now easily identify the reduction of $Z$ with the set of elements of $D$ of the form
\begin{center}
\begin{tikzpicture}[scale=.75]
	\node at (1,1) {$a$};

	\tikzset{font=\scriptsize};
	\draw (0,0) rectangle (2,2);
	\node at (-.6,1) {$1^{H}_m$};
	\node at (1,2.4) {$\wt s_{H}(a)$};
	\node at (2.7,1) {$\wt s_V(a)$};
	\node at (1,-.4) {$1^V_m$};
\end{tikzpicture}
\end{center}
which we might call the ``target-core'' $C_t$ of $D$ to distinguish it from the ``source-core'' $C_s\ (:= C)$ previously defined. However, the target-core $C_t$ can be identified with $C_s$ using the composition of the two groupoid inverses on $D$:
\[
\wt i_V \circ \wt i_{H}: C_t \to C_s,
\]
so we again get that the reduction of $Z$ is symplectomorphic to the core $C$. Considering the transpose of $D$, we find that the same is true for the preimage of the units of $H$ under the target of the left groupoid.

In summary, we have shown:
\begin{thrm}\label{thrm:core}
Let $D$ be a symplectic double groupoid with core $C$. Then the reductions of the coisotropic submanifolds
\[
\wt s_{H}^{-1}(1^V M),\ \wt s_V^{-1}(1^{H}M),\ \wt t_{H}^{-1}(1^V M), \text{ and } \wt t_V^{-1}(1^{H}M)
\]
of $D$ are symplectomorphic to $C$. 
\end{thrm}

\begin{ex}
Let us return to Example \ref{ex:main}. The target of the top groupoid is the same as the source, so the above reduction procedure produces the same reduction relation
\[
(T^*e)^t: T^*G \red T^*M
\]
as before. A similar computation to that carried out for the source of the left groupoid $T^*G \rightrightarrows A^*$ shows that the coisotropic submanifold $\wt \ell^{-1}(0_M)$ of $T^*G$ is $N^*\F_r$, where $\F_r$ is the foliation of $G$ given by the $r$-fibers of $G \rightrightarrows M$, and that the reduction relation
\[
T^*G \red T^*M
\]
obtained by reducing $N^*\F_r$ is then the cotangent lift $T^*r$.
\end{ex}

\begin{rmk}\label{gen-cot}
These examples show that in the case of the standard double groupoid structure on the cotangent bundle $T^*G$ of a Lie groupoid $G \rightrightarrows M$, the reduction relations arising from the various ways of realizing the core $T^*M$ are precisely the cotangent lifts of the structure maps of $G$. The reduction relations obtained by applying this procedure to a general cotangent double groupoid as in Theorem~\ref{thrm:ctdbl} can  be characterized as cotangent lifts of certain smooth relations.
\end{rmk}

\begin{ex}
Consider the double groupoid of Example~\ref{ex:inertia}. The coisotropic submanifold $X$ of $\overline{T^*G} \times T^*G$ is $\overline{T^*G} \times A^*$, whose reduction is $T^*G$ since $A^*$ is a Lagrangian submanifold of $T^*G$. The reduction relation $\overline{T^*G} \times T^*G \red T^*G$ is $id \times A^*$.

In the transposed double groupoid, the coisotropic submanifold $Y$ of $\overline{T^*G} \times T^*G$ is the fiber product of $\wt s: T^*G \to A^*$ with itself. Let $((g,\xi),(h,\eta))$ be an element of this fiber product. Then in particular $s(g) = s(h)$, so $gh^{-1}$ is defined under the multiplication on $G$. The leaf of the characteristic foliation of $Y$ containing $((g,\xi),(h,\eta))$ consists of elements of the form
$$((g,\xi)\circ(k,\omega),(h,\eta)\circ(k,\omega))$$
where $\circ$ is the product on $T^*G$ and $(k,\omega) \in T^*G$ satisfies $\wt t(k,\omega) = \wt s(g,\xi) = \wt s(h,\eta)$. The element of the core $T^*G$ associated with this leaf is
$$(g,\xi) \circ (h^{-1},(di)^*_{h^{-1}}\eta)$$
where $i$ is the inverse of $G$. As mentioned at the end of Remark~\ref{gen-cot}, the resulting reduction relation $\overline{T^*G} \times T^*G \red T^*G$ can be described as a cotangent lift; in particular, it is the cotangent lift of the \emph{smooth relation} $G \times G \to G$ given by $(g,h) \mapsto gh^{-1} \text{ for $g,h \in G$ such that } r(g)=r(h)$. (This smooth relation is related to the \emph{inertia} groupoid of $G$, which is the subgroupoid of $G$ consisting of loops, which are arrows whose target equals their source.)
\end{ex}

We thus have multiple ways of recovering the core of $D$ by reducing certain coisotropic submanifolds, and each such way produces a canonical relation $D \red C$.

\subsection{Symplectic Double Groups}
It is instructive to see the results of the above constructions in the special case of a symplectic double group. We see that we recover a result of Zakrzewski \cite{SZ2} concerning Hopf algebra objects in the symplectic category.

In a symplectic double group of the form:
\begin{center}
\begin{tikzpicture}[>=angle 90,scale=.85,baseline=(current  bounding  box.center)]
	\node at (0,0) {$P^*$};
	\node at (2,0) {$pt,$};
	\node at (0,2) {$S$};
	\node at (2,2) {$P$};

	\tikzset{font=\scriptsize};
	\draw[->] (.4,.08) -- (1.6,.08);
	\draw[->] (.4,-.08) -- (1.6,-.08);
	
	\draw[->] (-.08,1.6) -- (-.08,.4);
	\draw[->] (.08,1.6) -- (.08,.4);
	
	\draw[->] (.4,2+.08) -- (1.6,2+.08);
	\draw[->] (.4,2-.08) -- (1.6,2-.08);
	
	\draw[->] (2-.08,1.6) -- (2-.08,.4);
	\draw[->] (2+.08,1.6) -- (2+.08,.4);
\end{tikzpicture}
\end{center}
$P$ and $P^*$ are dual Poisson Lie groups. The coisotropic submanifold $X := \widetilde{s}_{P^*}^{-1}(1^{P} pt)$ of $S$ in this case is Lagrangian and can be identified with the embedding of $P^*$ into $S$. Being Lagrangian, its reduction is a point which is indeed the core of $S$. The reduction relation $S \red pt$ is given by $P^*$.

The transpose of this relation, together with the graph of the left groupoid product $S \times S \to S$ give $S$ the structure of a symplectic monoid. The canonical relation $S \red pt$ obtained by reducing $Y := \widetilde{s_P}^{-1}(1^{P^*}pt)$ is given by $P$, and along with the transpose of the top groupoid product gives $S$ the structure of a symplectic comonoid. The compatibility between these two structures can then be expressed by simply saying that $S$ together with these structures is a Hopf algebra object in $\Symp$, where the antipode $S \to S$ is the composition of the two groupoid inverses of $S$; this is then a generalization of Example \ref{ex:hopf}. In fact, we have the following:

\begin{thrm}[Zakrzewski, \cite{SZ2}]
$S$ endowed with two symplectic groupoid structures is a symplectic double group if and only if $S$ endowed with the aforementioned monoid and comonoid structures forms a Hopf algebra object in $\Symp$.
\end{thrm}

To be clear, what we call here a ``Hopf algebra object'' in the symplectic category is what Zakrzewski calls an ``$S^*$-group'', and his result is phrased in a slightly different manner.

\subsection{Core Groupoids}
We also point out that our reduction procedure gives a way of recovering Brown and Mackenzie's symplectic groupoid structure on the core $C$ of $D$.

\begin{prop}
The canonical relation $m: C \times C \to C$ given by the composition
\begin{displaymath}\label{core-product}
\begin{tikzpicture}[>=angle 90,baseline=(current  bounding  box.center)]
	\node (UL) at (0,0) {$C \times C$};
	\node (UR) at (3,0) {$D \times D$};
	\node (LL) at (6,0) {$D$};
	\node (LR) at (8,0) {$C,$};

	\tikzset{font=\scriptsize};
	\draw[->] (UL) to node [above] {$\mathbf e \times \mathbf e$} (UR);
	\draw[->] (UR) to node [above] {$\wt m_V$} (LL);
	\draw[->] (LL) to node [above] {$\mathbf e^t$} (LR);
\end{tikzpicture}
\end{displaymath}
where $\wt m_V$ is the product of the left grouopid $D \rightrightarrows H$ viewed as a canonical relation, is Brown and Mackenzie's groupoid product on $C$.
\end{prop}

\begin{proof}
Let $c, c' \in C$. Applying the relation $\mathbf e \times \mathbf e$ gives
\[
(c,c') \mapsto (c\,\wt\circ_{H}\,\wt 1^V_\lambda, c'\,\wt\circ_{H}\,\wt 1^V_{\lambda'})
\]
where $\lambda, \lambda'$ are as in~(\ref{unit}). Now, these are composable under $\wt m_V$ when
\[
\wt r_V(c\,\wt\circ_{H}\,\wt 1^V_\lambda) = \wt r_V(c) \circ_{H} \lambda = \lambda \text{ equals } \wt \ell_V(c'\,\wt\circ_{H}\,\wt 1^V_{\lambda'}) = \wt\ell_V(c') \circ_{H} \lambda',
\]
where we have used the fact that $\wt r_V(c)$ is a unit. From this we get the condition that
\[
\lambda = \wt\ell_V(c') \circ_{H}\lambda'.
\]
The relation $\wt m_V$ then produces
\[
(c\,\wt\circ_{H}\,\wt 1^V_\lambda) \wt\circ_V (c'\,\wt\circ_{H}\,\wt 1^V_{\lambda'}) = \left(c\,\wt\circ_{H}\,\wt 1^V_{\wt\ell_V(c') \circ_{H}\lambda'}\right) \wt\circ_V (c'\,\wt\circ_{H}\,\wt 1^V_{\lambda'}).
\]

Now, applying the map $\wt r_H$ to this gives
\[ \wt r_H(\wt 1^V_{\wt\ell_V(c') \circ_{H}\lambda'})\,\wt \circ_V\,\wt r_H(\wt 1^V_{\lambda'}) = 1^V_{r_H(\lambda')}\,\circ_V\,1^V_{r_H(\lambda')} = 1^V_{r_H(\lambda')}, \]
so that the element above is already in the domain $X$ of $\mathbf e^t$. Applying the final relation $\mathbf e^t$ then gives
\[ \left[\left(c\,\wt\circ_{H}\,\wt 1^V_{\wt\ell_V(c') \circ_{H}\lambda'}\right) \wt\circ_V (c'\,\wt\circ_{H}\,\wt 1^V_{\lambda'})\right]\,\wt\circ_H\,\wt 1^V_{i_H(\lambda')}. \]
Expressing this in square notation, we have:
\begin{center}
\begin{tikzpicture}
	\draw (0,0) rectangle (2.75,1);
	\node at (1.375,.5) {$c\,\wt\circ_{H}\,\wt 1^V_{\wt\ell_V(c') \circ_{H}\lambda'}$};
	\draw (2.75,0) rectangle (4.5,1);
	\node at (3.625,.5) {$c'\,\wt\circ_{H}\,\wt 1^V_{\lambda'}$};
	\draw (0,1) rectangle (4.5,2);
	\node at (2.25,1.5) {$\wt 1^V_{i_H(\lambda')}$};
	
	\node at (5,1) {$=$};
	
	\draw (5.5,0) rectangle (8.25,1);
	\node at (6.875,.5) {$c\,\wt\circ_{H}\,\wt 1^V_{\wt\ell_V(c') \circ_{H}\lambda'}$};
	\draw (8.25,0) rectangle (10.25,1);
	\node at (9.25,.5) {$c'\,\wt\circ_{H}\,\wt 1^V_{\lambda'}$};
	\draw (5.5,1) rectangle (8.25,2);
	\node at (6.875,1.5) {$\wt 1^V_{i_H(\lambda')}$};
	\draw (8.25,1) rectangle (10.25,2);
	\node at (9.25,1.5) {$\wt 1^V_{i_H(\lambda')}$};
	
	\node at (10.75,1) {$=$};
	
	\draw (11.25,.5) rectangle (13.25,1.5);
	\node at (12.25,1) {$c\,\wt\circ_H\,\wt 1^V_{\wt\ell_V(c')}$};
	\draw (13.25,.5) rectangle (14.25,1.5);
	\node at (13.75,1) {$c'$};
\end{tikzpicture}
\end{center}
where in the first step we decompose the top box horizontally and in the second we compose vertically. Using $c' = \wt 1^H_{\ell_H(c')}\,\wt\circ_H\, c'$, we can write this final expression as
\begin{center}
\begin{tikzpicture}
	\draw (-4,.5) rectangle (-2,1.5);
	\node at (-3,1) {$c\,\wt\circ_H\,\wt 1^V_{\wt\ell_V(c')}$};
	\draw (-2,.5) rectangle (-1,1.5);
	\node at (-1.5,1) {$c'$};
	
	\node at (-.5,1) {$=$};
	
	\draw (0,0) rectangle (1.5,1);
	\node at (.75,.5) {$c$};
	\draw (1.5,0) rectangle (3,1);
	\node at (2.25,.5) {$\wt 1^{H}_{\wt\ell_{H}(c')}$};
	\draw (0,1) rectangle (1.5,2);
	\node at (.75,1.5) {$\wt 1^V_{\wt\ell_V(c')}$};
	\draw (1.5,1) rectangle (3,2);
	\node at (2.25,1.5) {$c'$};
\end{tikzpicture}
\end{center}
which gives after composing either horizontally first and then vertically, or vertically first and then horizontally, Brown and Mackenzie's core groupoid product as claimed.
\end{proof}

The unit $e: pt \to C$ defined as the composition
\begin{displaymath}\label{core-unit}
\begin{tikzpicture}[>=angle 90,baseline=(current  bounding  box.center)]
	\node (UL) at (0,1) {$pt$};
	\node (UM) at (2,1) {$D$};
	\node (UR) at (4,1) {$C,$};

	\tikzset{font=\scriptsize};
	\draw[->] (UL) to node [above] {$H$} (UM);
	\draw[->] (UM) to node [above] {$\mathbf e^t$} (UR);
\end{tikzpicture}
\end{displaymath}
where $H$ denotes the image of the Lagrangian embedding $\widetilde 1_V: H \to D$, and $*$-structure $s: \overline C \to C$ defined as the composition
\begin{displaymath}\label{core-star}
\begin{tikzpicture}[>=angle 90,baseline=(current  bounding  box.center)]
	\node (1) at (0,1) {$\overline{C}$};
	\node (2) at (2,1) {$\overline{D}$};
	\node (3) at (4,1) {$D$};
	\node (4) at (6,1) {$C$};

	\tikzset{font=\scriptsize};
	\draw[->] (1) to node [above] {$\mathbf e$} (2);
	\draw[->] (2) to node [above] {$\wt i_V$} (3);
	\draw[->] (3) to node [above] {$\mathbf e^t$} (4);
\end{tikzpicture}
\end{displaymath}
then endow $(C,m,e)$ with the structure of a strongly positive symplectic $*$-monoid in the symplectic category, which is indeed a symplectic groupoid according to Zakrzewski. The various compatibilities required in the definition of a strongly positive symplectic $*$-monoid follow from commutativity of the diagrams
\begin{center}
\begin{tikzpicture}[>=angle 90]
	\node (L1) at (0,0) {$C$};
	\node (L2) at (3,0) {$D$};
	\node (U1) at (0,2) {$C \times C$};
	\node (U2) at (3,2) {$D \times D$};
	\node (U3) at (6,1) {$D$};
	\node (U4) at (9,1) {$C$};

	\tikzset{font=\scriptsize};
	\draw[->] (U1) to node [above] {$\mathbf e \times \mathbf e$} (U2);
	\draw[->] (U2) to node [above] {$\wt m_V$} (U3);
	\draw[->] (U3) to node [above] {$\mathbf e^t$} (U4);
	\draw[->] (L1) to node [above] {$\mathbf e$} (L2);
	\draw[->] (L1) to node [left] {$e \times id$} (U1);
	\draw[->] (L2) to node [left] {$H \times id$} (U2);
	\draw[->] (L2) to node [below] {$id$} (U3);
	\draw[->] (U1) to [bend left=30] node [above] {$m$} (U4);
	\draw[->] (L1) to [bend right=20] node [below] {$id$} (U4);
\end{tikzpicture}
\end{center}
for the left unit property (a similar one applies to the right unit property), and
\begin{center}
\begin{tikzpicture}[>=angle 90,scale=.75]
	\node (ULC) at (0,3) {$\overline{C} \times \overline{C}$};
	\node (UC) at (3,3) {$\overline{C} \times \overline{C}$};
	\node (URC) at (7,3) {$C \times C$};
	\node (LLC) at (0,-3) {$\overline{C}$};
	\node (LRC) at (7,-3) {$C$};

	\node (UL) at (0,1) {$\overline{D} \times \overline{D}$};
	\node (U) at (3,1) {$\overline{D} \times \overline{D}$};
	\node (UR) at (7,1) {$D \times D$};
	\node (LL) at (0,-1) {$\overline{D}$};
	\node (LR) at (7,-1) {$D$};

	\tikzset{font=\scriptsize};
	\draw[->] (ULC) to node [above] {$\sigma$} (UC);
	\draw[->] (UC) to node [above] {$s \times s$} (URC);
	\draw[->] (ULC) to [bend right=80] node [left] {$m$} (LLC);
	\draw[->] (URC) to [bend left=80] node [right] {$m$} (LRC);
	\draw[->] (LLC) to node [above] {$s$} (LRC);

	\draw[->] (UL) to node [above] {$\sigma$} (U);
	\draw[->] (U) to node [above] {$\wt i_V \times \wt i_V$} (UR);
	\draw[->] (UL) to node [left] {$\wt m_V$} (LL);
	\draw[->] (UR) to node [right] {$\wt m_V$} (LR);
	\draw[->] (LL) to node [above] {$\wt i_V$} (LR);
	
	\draw[->] (ULC) to node [left] {$\mathbf e \times \mathbf e$} (UL);
	\draw[->] (LL) to node [left] {$\mathbf e^t$} (LLC);
	\draw[->] (URC) to node [right] {$\mathbf e \times \mathbf e$} (UR);
	\draw[->] (LR) to node [right] {$\mathbf e^t$} (LRC);
\end{tikzpicture}
\end{center}
where $\sigma$ is the symplectomorphism exchanging components, for the strongly positive $*$-structure property. We omit the full details of these verifications.

\begin{rmk}
Now that we have obtained the symplectic groupoid structure on the core $C \rightrightarrows M$ of $D$ via reduction, we know that $M$ inherits a unique Poisson structure relative to which the source map of the core is Poisson and the target map anti-Poisson. It would be interesting to know whether this Poisson structure can be derived more directly from the reduction procedure without making use of the full symplectic groupoid structure on $C$.
\end{rmk}

\section{Symplectic Hopfoids}
\subsection{From Double Groupoids to Hopfoids}
Let us return to the structures of Example \autoref{ex:main}. To recall, performing the procedure above in this case produced the relations $(T^*e)^t: T^*G \to T^*M$, $T^*s: T^*G \to T^*M$, and $T^*t: T^*G \to T^*r$ obtained by taking the cotangent lifts of the groupoid structure maps of $G \rightrightarrows M$. Also note that the cotangent lift $T^*i: T^*G \to T^*G$ of the groupoid inverse of $G$ is just the composition of the two symplectic groupoid inverses on $T^*G$. In addition, the canonical relations $T^*m: T^*G \times T^*G \to T^*G$ and $T^*\Delta: T^*G \to T^*G \times T^*G$ (where $\Delta: G \to G \times G$ is the usual diagonal map) come from the two groupoid products on $T^*G$.

The above canonical relations are precisely the ones which appear in the diagrams obtained by applying the cotangent functor $T^*$ to the commutative diagrams appearing in the definition of a Lie groupoid. This suggests that the structure $T^*G \rightrightarrows T^*M$ obtained should be viewed as the analog of a ``groupoid'' in the symplectic category, as previously mentioned.

This all generalizes in the following way for an arbitrary symplectic double groupoid; we will denote the resulting structure by $D \rightrightarrows C$. First, the coreduction $\mathbf{e}: C \to D$ is given in (\ref{unit}); this is the analog of $T^*e: T^*M \cored T^*G$ in Example \ref{ex:main}. Second, the reduction $\mathbf{t}: D \to C$ is given in (\ref{target}); this is the analog of $T^*t: T^*G \red T^*M$ in Example \ref{ex:main}. We will think of $\mathbf{e}$ as the ``unit'' and $\mathbf{t}$ the ``target'' of $D \rightrightarrows C$.

Now, as noted before, the reduction of $\wt t_V^{-1}(1^{H}M)$ naturally gives the target-core $C_t$, so to get a morphism to $C$ we must post-compose the resulting reduction relation $D \red C_t$ with the composition $\wt i_V \circ \wt i_{H}$ of the two inverses on $D$. Note that since the inverse of a symplectic groupoid is an anti-symplectomorphism, this composition of two inverses is a symplectomorphism, so the relation $\mathbf{s}: D \red C$ so obtained is indeed a canonical relation. This relation is explicitly given by
\begin{equation}\label{source}
\mathbf{r}: a \mapsto \wt i_V \wt i_{H}(a)\ \wt\circ_V\ \wt 1^{H}_{\wt t_{H}(a)}.
\end{equation}
In Example \ref{ex:main} this relation became $T^*s: T^*G \red T^*M$, and we will think of $\mathbf{s}$ as the ``source'' of $D \rightrightarrows C$.

Finally, we will let $\mathbf{i}: D \to D$ be the composition $\wt i_V \circ \wt i_{H}$, which as noted above is a symplectomorphism; in Example \ref{ex:main} this is $T^*i$. We then have the following observations:

\begin{prop}\label{prop:sections}
The canonical relations above satisfy the following identities: $\mathbf{t} \circ \mathbf{e} = id_C$, $\mathbf{s} \circ \mathbf{e} = id_C$, $\mathbf{s} = \mathbf{r} \circ \mathbf{i}$, $\mathbf{t} = \mathbf{s} \circ \mathbf{i}$. All of these compositions are strongly transversal.
\end{prop}

\begin{proof}
Suppose that $a \in C$, so that $\wt s_V(a)$ and $\wt s_{H}(a)$ are both units. Then the relation $\mathbf{e}$ sends this to 
\[
\mathbf{e}: a \mapsto a\,\wt\circ_{H}\,\wt 1^V_\lambda
\]
for $\lambda$ in the same leaf as $a$ of the characteristic foliation of $\wt s_{H}^{-1}(1^V M)$. Now, this is in the domain of $\mathbf{t}$ when
\[
\wt s_V(a\,\wt\circ_{H}\,\wt 1^V_\lambda) = \wt s_V(a)\,\circ_{H}\,\lambda = \lambda
\]
is a unit. Thus there is only one such $\lambda$ so that $a\,\wt\circ_{H}\,\wt 1^V_\lambda \in \dom\mathbf{t}$, from which strong transversality will follow, and it is then straightforward to check that applying the relation $\mathbf{t}$ will give back $a$. Hence $\mathbf{t} \circ \mathbf{e} = id_C$.

More interestingly, $a\,\wt\circ_{H}\,\wt 1^V_\lambda$ is in the domain of $\mathbf{s}$ when
\[
\wt t_V(a\,\wt\circ_{H}\,\wt 1^V_\lambda) = \wt t_V(a)\circ_{H}\lambda
\]
is a unit, which requires that $\lambda = i_{H}(\wt t_V(a))$. Again, from this strong transversality follows and we see that the composition $\mathbf{s} \circ \mathbf{e}$ is
\[
a \mapsto \wt i_V \wt i_{H}\left(a\,\wt\circ_{H}\,\wt 1^V_{i_{H}(\wt t_V(a))}\right) \wt\circ_V\,\wt 1^{H}_{\wt t_{H}(a)} = \left(\wt 1^V_{\wt t_V(a)}\,\wt\circ_{H}\,\wt i_V\wt i_{H}(a)\right)\wt\circ_V\,\wt 1^{H}_{\wt t_{H}(a)}.
\]

We claim that this result is simply $a$. To see this, we express the result as
\begin{center}
\begin{tikzpicture}[>=angle 90]
	\draw (0,0) rectangle (2,2);
	\node at (1,1) {$\wt 1^V_{\wt t_V(a)}$};
	\draw (0,2) rectangle (2,4);
	\node at (1,3) {$\wt i_V\wt i_{H}(a)$};
	\draw (2,0) rectangle (4,4);
	\node at (3,2) {$\wt 1^{H}_{\wt t_{H}(a)}$};
	
	\node at (5,2) {$=$};
	
	\draw (6,0) rectangle (8,2);
	\node at (7,1) {$\wt 1^V_{\wt t_V(a)}$};
	\draw (8,0) rectangle (10,2);
	\node at (9,1) {$a$};
	\draw (6,2) rectangle (8,4);
	\node at (7,3) {$\wt i_V\wt i_{H}(a)$};
	\draw (8,2) rectangle (10,4);
	\node at (9,3) {$\wt i_{H}(a)$};
	
	\node at (11,2) {$=$};
	
	\draw (12,0) rectangle (14,2);
	\node at (13,1) {$a$};
	\draw (12,2) rectangle (14,4);
	\node at (13,3) {$\wt 1^V_{i_{H}(\wt s_V(a))}$};
\end{tikzpicture}
\end{center}
where in the first step we decompose vertically and in the second we compose horizontally. Now, since $\wt s_V(a)$ is a unit, the element on top in the last term is of the form
\[
\wt 1^V_{i_{H}(1^{H}_m)} = \wt 1^V 1^{H}_m = \wt 1^{H}1^V_m.
\]
The claim then follows by composing the last term vertically. The computations which show that $\mathbf{s} = \mathbf{t} \circ \mathbf{i}$ and $\mathbf{t} = \mathbf{s} \circ \mathbf{i}$ and that these compositions are strongly transversal are similar.
\end{proof}

Now, there are two groupoid multiplications on $D$. We will denote by $\mathbf{m}: S \times S \to S$ the canonical relation given by $\wt m_V$, and will think of $\mathbf{m}$ as a ``product''. Similarly, we will denote by $\Delta: S \to S \times S$ the transpose of canonical relation obtained from $\wt m_{H}$, and will think of $\Delta$ as a ``coproduct''.

The following propositions then express some compatibilities between these relations and those previously defined. In particular, in the example of $T^*G \rightrightarrows T^*M$, we emphasize that the diagrams considered are precisely those obtained by applying the cotangent functor to those in the diagrams expressing the compatibilities between the structure maps of the Lie groupoid $G \rightrightarrows M$.

\begin{prop}
The following diagram commutes and all compositions are strongly transversal:
\begin{center}
\begin{tikzpicture}[>=angle 90]
	\node (UL) at (0,1) {$D \times D$};
	\node (UM) at (3,1) {$C \times D$};
	\node (UR) at (6,1) {$D \times D$};
	\node (LL) at (0,-1) {$D$};
	\node (LR) at (6,-1) {$D$};

	\tikzset{font=\scriptsize};
	\draw[->] (UL) to node [above] {$\mathbf{t} \times id$} (UM);
	\draw[->] (UM) to node [above] {$\mathbf{e} \times id$} (UR);
	\draw[->] (LL) to node [left] {$\Delta$} (UL);
	\draw[->] (UR) to node [right] {$\mathbf{m}$} (LR);
	\draw[->] (LL) to node [above] {$id$} (LR);
\end{tikzpicture}
\end{center}
\end{prop}

\begin{proof}
For $a \in D$, the composition $\mathbf{m} \circ (\mathbf{e} \times id) \circ (\mathbf{t} \times id) \circ \Delta$ is
\[
a \mapsto [(a_1\ \wt\circ_V\ \wt 1^{H}_{i_V(\wt s_{H}(a_1))})\ \wt\circ_{H}\ \wt 1^V_\lambda]\ \wt\circ_V\ a_2
\]
for $\lambda$ in the same leaf as $a$ of the characteristic foliation of $\wt s_{H}^{-1}(1^V M)$ and where $a = a_1\ \wt\circ_{H}\ a_2$ for composable---with respect to the groupoid structure whose product is $\Delta^t$---$a_1$ and $a_2 \in D$; this is simply saying that $a$ is in the domain of $\Delta$. We must show that the resulting expression is just $a$ itself. This follows from the compositions:

\begin{center}
\begin{tikzpicture}
	\draw (0,0) rectangle (2,2);
	\node at (1,1) {$a_1$};
	\draw (2,0) rectangle (4,2);
	\node at (3,1) {$\,\wt 1^{H}_{i_V(\wt s_{H}(a_1))}$};
	\draw (0,2) rectangle (4,4);
	\node at (2,3) {$\wt 1^V_\lambda$};
	\draw (4,0) rectangle (6,4);
	\node at (5,2) {$a_2$};
	
	\node at (7,2) {$=$};
	
	\draw (8,0) rectangle (10,2);
	\node at (9,1) {$a_1$};
	\draw (10,0) rectangle (12,2);
	\node at (11,1) {$\,\wt 1^{H}_{i_V(\wt s_{H}(a_1))}$};
	\draw (12,0) rectangle (14,2);
	\node at (13,1) {$\wt 1^{H}_{\wt t_{H}(a_2)}$};
	\draw (8,2) rectangle (10,4);
	\node at (9,3) {$a_2$};
	\draw (10,2) rectangle (12,4);
	\node at (11,3) {$\wt i_V(a_2)$};
	\draw (12,2) rectangle (14,4);
	\node at (13,3) {$a_2$};
\end{tikzpicture}
\end{center}
where we decompose the top left box horizontally and the right box vertically, and

\begin{center}
\begin{tikzpicture}[scale=.91]
	\draw (8,0) rectangle (10,2);
	\node at (9,1) {$a_1$};
	\draw (10,0) rectangle (12,2);
	\node at (11,1) {$\,\wt 1^{H}_{i_V(\wt s_{H}(a_1))}$};
	\draw (12,0) rectangle (14,2);
	\node at (13,1) {$\wt 1^{H}_{\wt t_{H}(a_2)}$};
	\draw (8,2) rectangle (10,4);
	\node at (9,3) {$a_2$};
	\draw (10,2) rectangle (12,4);
	\node at (11,3) {$\wt i_V(a_2)$};
	\draw (12,2) rectangle (14,4);
	\node at (13,3) {$a_2$};
	
	\node at (15,2) {$=$};
	
	\draw (16,0) rectangle (18,2);
	\node at (17,1) {$a_1$};
	\draw (18,0) rectangle (20,2);
	\node at (19,1) {$\,\wt 1^{H}_{s_V(\wt t_{H}(a_2))}$};
	\draw (16,2) rectangle (18,4);
	\node at (17,3) {$a_2$};
	\draw (18,2) rectangle (20,4);
	\node at (19,3) {$\wt 1^V_{\wt s_V(a_2)}$};
	
	\node at (21,2) {$=$};
	
	\draw (22,0) rectangle (24,2);
	\node at (23,1) {$a_1$};
	\draw (22,2) rectangle (24,4);
	\node at (23,3) {$a_2$};
	\draw (24,0) rectangle (26,4);
	\node at (25,2) {$\wt 1^V_{\wt s_V(a_2)}$};
\end{tikzpicture}
\end{center}
where we first compose the right four boxes horizontally (using the fact that $\wt s_{H}(a_1) = \wt t_{H}(a_2)$) and then compose the right two boxes vertically. The resulting composition is then $a_1\,\wt\circ_{H}\,a_2 = a$ as was to be shown.
\end{proof}

The proof of the following result is very similar to that above and is omitted. The need to use $\mathbf{i} \circ \mathbf{e}$ in place of simply $\mathbf{e}$ here comes from the difference between the target- and source-cores of $D$.

\begin{prop}
The following diagram commutes and all compositions are strongly transversal:
\begin{center}
\begin{tikzpicture}[>=angle 90]
	\node (UL) at (0,1) {$D \times D$};
	\node (UM) at (3,1) {$D \times C$};
	\node (UR) at (6,1) {$D \times D$};
	\node (LL) at (0,-1) {$D$};
	\node (LR) at (6,-1) {$D$};

	\tikzset{font=\scriptsize};
	\draw[->] (UL) to node [above] {$id \times \mathbf{s}$} (UM);
	\draw[->] (UM) to node [above] {$id \times (\mathbf{i} \circ \mathbf{e})$} (UR);
	\draw[->] (LL) to node [left] {$\Delta$} (UL);
	\draw[->] (UR) to node [right] {$\mathbf{m}$} (LR);
	\draw[->] (LL) to node [above] {$id$} (LR);
\end{tikzpicture}
\end{center}
\end{prop}

The following brings in an antipode/inverse-like condition on $\mathbf{i}$.

\begin{prop}
The following diagram commutes and all compositions are strongly transversal:
\begin{center}
\begin{tikzpicture}[>=angle 90]
	\node (UL) at (0,1) {$D \times D$};
	\node (UR) at (6,1) {$D \times D$};
	\node (LL) at (0,-1) {$D$};
	\node (LM) at (3,-1) {$C$};
	\node (LR) at (6,-1) {$D$};

	\tikzset{font=\scriptsize};
	\draw[->] (UL) to node [above] {$id \times \mathbf{i}$} (UR);
	\draw[->] (LL) to node [left] {$\Delta$} (UL);
	\draw[->] (UR) to node [right] {$\mathbf{m}$} (LR);
	\draw[->] (LL) to node [above] {$\mathbf{t}$} (LM);
	\draw[->] (LM) to node [above] {$\mathbf{e}$} (LR);
\end{tikzpicture}
\end{center}
\end{prop}

\begin{proof}
Let $a \in D$. The composition $\mathbf{m} \circ (id \times \mathbf{i}) \circ \Delta$ looks like
\[
a \mapsto a_1\ \wt\circ_V\ \wt i_V \wt i_{H}(a_2)
\]
where $a = a_1 \wt\circ_{H} a_2$. The composition $\mathbf{e} \circ \mathbf{t}$ is
\[
a \mapsto (a\ \wt\circ_V\ \wt 1^{H}_{i_V(\wt s_{H}(a))})\ \wt\circ_{H}\ \wt 1^V_\lambda
\]
for $\lambda$ in the same leaf as $a$ of the characteristic foliation of $\wt s_{H}^{-1}(1^V M)$. We must show that anything of the form resulting from the first relation is equivalently of the form resulting from the second. In particular, writing $a$ as $a = a_1 \wt\circ_{H} a_2$, we must show that
\[
a_1\ \wt\circ_V\ \wt i_V \wt i_{H}(a_2) =  [(a_1 \wt\circ_{H} a_2)\ \wt\circ_V\ \wt 1^{H}_{i_V(\wt s_{H}(a_2))}]\ \wt\circ_{H}\ \wt 1^V_\lambda.
\]

For this we proceed as follows. The expression on the right is
\begin{center}
\begin{tikzpicture}
	\draw (0,0) rectangle (2,2);
	\node at (1,1) {$a_1$};
	\draw (0,2) rectangle (2,4);
	\node at (1,3) {$a_2$};
	\draw (2,0) rectangle (4,4);
	\node at (3,2) {$\,\wt 1^{H}_{i_V(\wt s_{H}(a_2))}$};
	\draw (0,4) rectangle (4,6);
	\node at (2,5) {$\wt 1^V_\lambda$};
	
	\node at (5,3) {$=$};
	
	\draw (6,0) rectangle (8,2);
	\node at (7,1) {$a_1$};
	\draw (6,2) rectangle (8,4);
	\node at (7,3) {$a_2$};
	\draw (8,0) rectangle (10,2);
	\node at (9,1) {$\,\wt 1^{H}_{i_V(\wt s_{H}(a_2))}$};
	\draw (8,2) rectangle (10,4);
	\node at (9,3) {$\,\wt 1^{H}_{i_V(\wt s_{H}(a_2))}$};
	\draw (6,4) rectangle (8,6);
	\node at (7,5) {$\wt i_{H}(a_2)$};
	\draw (8,4) rectangle (10,6);
	\node at (9,5) {$\wt i_V\wt i_{H}(a_2)$};
	
	\node at (11,3) {$=$};
	
	\draw (12,1) rectangle (14,3);
	\node at (13,2) {$a_1$};
	\draw (14,1) rectangle (16,3);
	\node at (15,2) {$\,\wt 1^{H}_{i_V(\wt s_{H}(a_2))}$};
	\draw (12,3) rectangle (14,5);
	\node at (13,4) {$\wt 1^{H}_{\wt t_{H}(a_2)}$};
	\draw (14,3) rectangle (16,5);
	\node at (15,4) {$\wt i_V\wt i_{H}(a_2)$};
\end{tikzpicture}
\end{center}
where in the first step we have decomposed the right box vertically and the top box horizontally, and in the second we have composed the top four boxes vertically. The desired equality now follows by composing the remaining boxes vertically.

Strong transversality in $\mathbf{m} \circ (id \times \mathbf{i}) \circ \Delta$ follows from $\mathbf{i}$ being a symplectomorphism, and in $\mathbf{e} \circ \mathbf{t}$ from $\mathbf{t}$ being a reduction, or $\mathbf{e}$ being a coreduction.
\end{proof}

Again, the proof of the following is very similar to that of the above proposition and is omitted.

\begin{prop}
The following diagram commutes and all compositions are strongly transversal:
\begin{center}
\begin{tikzpicture}[>=angle 90]
	\node (UL) at (0,1) {$D \times D$};
	\node (UR) at (6,1) {$D \times D$};
	\node (LL) at (0,-1) {$D$};
	\node (LM) at (3,-1) {$C$};
	\node (LR) at (6,-1) {$D$};

	\tikzset{font=\scriptsize};
	\draw[->] (UL) to node [above] {$\mathbf{i} \times id$} (UR);
	\draw[->] (LL) to node [left] {$\Delta$} (UL);
	\draw[->] (UR) to node [right] {$\mathbf{m}$} (LR);
	\draw[->] (LL) to node [above] {$\mathbf{s}$} (LM);
	\draw[->] (LM) to node [above] {$\mathbf{i} \circ \mathbf{e}$} (LR);
\end{tikzpicture}
\end{center}
\end{prop}

\subsection{Symplectic Hopfoids}
As stated before, the results of the previous propositions suggest that the structure $D \rightrightarrows C$ resulting from a symplectic double groupoid is similar to that of a ``groupoid''. This motivates the following definition:

\begin{defn}\label{defn:symp-hopf}
A \emph{symplectic hopfoid} $D \rightrightarrows C$ consists of the following data:
\begin{itemize}
\item a strongly positive symplectic $*$-comonoid $D$ and a symplectic submanifold $C \subseteq D$,
\item reductions $\mathbf{t}, \mathbf{s}: D \red C$ called the target and source respectively,
\item a coreduction $\mathbf{e}: C \cored D$ called the unit,
\item a canonical relation $\mathbf{m}: D \times D \to D$ called the product, and
\item a symplectomorphism $\mathbf{i}: D \to D$ called the antipode which preserves the counit of $D$
\end{itemize}
satisfying the following requirements:
\begin{enumerate}[(i)]
\item $\mathbf{t} \circ \mathbf{e} = id_C = \mathbf{s} \circ \mathbf{e}$,
\item $\mathbf{t} \circ \mathbf{i} = \mathbf{s}$ and $\mathbf{s} \circ \mathbf{i} = \mathbf{t}$,
\item (associativity) the diagram
\begin{center}
\begin{tikzpicture}[>=angle 90]
	\node (UL) at (0,1) {$D \times D \times D$};
	\node (UR) at (5,1) {$D \times D$};
	\node (LL) at (0,-1) {$D \times D$};
	\node (LR) at (5,-1) {$D$};

	\tikzset{font=\scriptsize};
	\draw[->] (UL) to node [above] {$id \times \mathbf{m}$} (UR);
	\draw[->] (UL) to node [left] {$\mathbf{m} \times id$} (LL);
	\draw[->] (UR) to node [right] {$\mathbf{m}$} (LR);
	\draw[->] (LL) to node [above] {$\mathbf{m}$} (LR);
\end{tikzpicture}
\end{center}
commutes,
\item (antipode) $\mathbf{i}^2 = id_D$, $\mathbf{i}$ commutes with the $*$-structure of $D$, and the diagram
\begin{center}\label{diag:antipode}
\begin{tikzpicture}[>=angle 90]
	\node (U1) at (0,1) {$D \times D$};
	\node (U2) at (3,1) {$D \times D$};
	\node (U3) at (6,1) {$D \times D$};
	\node (L1) at (0,-1) {$D$};
	\node (L3) at (6,-1) {$D$};

	\tikzset{font=\scriptsize};
	\draw[->] (U1) to node [above] {$\sigma$} (U2);
	\draw[->] (U2) to node [above] {$\mathbf{i} \times \mathbf{i}$} (U3);
	\draw[->] (U1) to node [left] {$\mathbf{m}$} (L1);
	\draw[->] (U3) to node [right] {$\mathbf{m}$} (L3);
	\draw[->] (L1) to node [above] {$\mathbf{i}$} (L3);
\end{tikzpicture}
\end{center}
where $\sigma: D \times D \to D \times D$ is the symplectomorphism exchanging components, commutes,
\item (compatibility between product and coproduct) the diagram\label{diag:product}
\begin{center}
\begin{tikzpicture}[>=angle 90]
	\node (U1) at (0,1) {$D \times D$};
	\node (U2) at (3,1) {$D$};
	\node (U3) at (6,1) {$D \times D$};
	\node (L1) at (0,-1) {$D \times D \times D \times D$};
	\node (L3) at (6,-1) {$D \times D \times D \times D$};

	\tikzset{font=\scriptsize};
	\draw[->] (U1) to node [above] {$\mathbf{m}$} (U2);
	\draw[->] (U2) to node [above] {$\Delta$} (U3);
	\draw[->] (U1) to node [left] {$\Delta \times \Delta$} (L1);
	\draw[->] (L3) to node [right] {$\mathbf{m} \times \mathbf{m}$} (U3);
	\draw[->] (L1) to node [above] {$id \times \sigma \times id$} (L3);
\end{tikzpicture}
\end{center}
where $\sigma: D \times D \to D \times D$ is the symplectomorphism exchanging components, commutes,
\item (left and right units) the diagrams\label{diag:units}
\begin{center}
\begin{tikzpicture}[>=angle 90]
	\node (UL) at (0,1) {$D \times D$};
	\node (UM) at (2.5,1) {$C \times D$};
	\node (UR) at (5,1) {$D \times D$};
	\node (LL) at (0,-1) {$D$};
	\node (LR) at (5,-1) {$D$};
	
	\node at (6.5,0) {$\text{and}$};
	
	\node (UL2) at (8,1) {$D \times D$};
	\node (UM2) at (10.5,1) {$D \times C$};
	\node (UR2) at (13.5,1) {$D \times D$};
	\node (LL2) at (8,-1) {$D$};
	\node (LR2) at (13.5,-1) {$D$};

	\tikzset{font=\scriptsize};
	\draw[->] (UL) to node [above] {$\mathbf{t} \times id$} (UM);
	\draw[->] (UM) to node [above] {$\mathbf{e} \times id$} (UR);
	\draw[->] (LL) to node [left] {$\Delta$} (UL);
	\draw[->] (UR) to node [right] {$\mathbf{m}$} (LR);
	\draw[->] (LL) to node [above] {$id$} (LR);
	
	\draw[->] (UL2) to node [above] {$id \times \mathbf{s}$} (UM2);
	\draw[->] (UM2) to node [above] {$id \times (\mathbf{i} \circ \mathbf{e})$} (UR2);
	\draw[->] (LL2) to node [left] {$\Delta$} (UL2);
	\draw[->] (UR2) to node [right] {$\mathbf{m}$} (LR2);
	\draw[->] (LL2) to node [above] {$id$} (LR2);
\end{tikzpicture}
\end{center}
commute, and
\item (left and right inverses) the diagrams\label{diag:inverses}
\begin{center}
\begin{tikzpicture}[>=angle 90]
	\node (UL) at (0,1) {$D \times D$};
	\node (UR) at (5,1) {$D \times D$};
	\node (LL) at (0,-1) {$D$};
	\node (LM) at (2.5,-1) {$C$};
	\node (LR) at (5,-1) {$D$};
	
	\node at (6.5,0) {$\text{and}$};
	
	\node (UL2) at (8,1) {$D \times D$};
	\node (UR2) at (13,1) {$D \times D$};
	\node (LL2) at (8,-1) {$D$};
	\node (LM2) at (10.5,-1) {$C$};
	\node (LR2) at (13,-1) {$D$};

	\tikzset{font=\scriptsize};
	\draw[->] (UL) to node [above] {$id \times \mathbf{i}$} (UR);
	\draw[->] (LL) to node [left] {$\Delta$} (UL);
	\draw[->] (UR) to node [right] {$\mathbf{m}$} (LR);
	\draw[->] (LL) to node [above] {$\mathbf{t}$} (LM);
	\draw[->] (LM) to node [above] {$\mathbf{e}$} (LR);
	
	\draw[->] (UL2) to node [above] {$\mathbf{i} \times id$} (UR2);
	\draw[->] (LL2) to node [left] {$\Delta$} (UL2);
	\draw[->] (UR2) to node [right] {$\mathbf{m}$} (LR2);
	\draw[->] (LL2) to node [above] {$\mathbf{s}$} (LM2);
	\draw[->] (LM2) to node [above] {$\mathbf{i} \circ \mathbf{e}$} (LR2);
\end{tikzpicture}
\end{center}
commute,
\end{enumerate}
where $\Delta := \Delta_D$ is the coproduct of the comonoid structure on $D$. We require that all compositions above be strongly transversal.
\end{defn}

\begin{rmk}
As mentioned previously, the ``hopf'' in the term hopfoid comes from the requirement that the coproduct and counit be specified as part of the data. Also, the requirement that $\mathbf t$ and $\mathbf s$ be reductions actually implies that they are partially-defined smooth maps.
\end{rmk}

In summary then, we have the following:

\begin{thrm}
Let $D$ be a symplectic double groupoid. Then the object $D \rightrightarrows C$ resulting from the constructions of the previous sections is a symplectic hopfoid.
\end{thrm}

\begin{proof}
The remaining conditions are straightforward to check; in particular, (vi) is precisely the compatibility between the two groupoid structures on $D$.
\end{proof}

\subsection{Back to General Double Groupoids}
Here we note that the above constructions can also be carried out even when there is no symplectic structure present---i.e. for a general double Lie groupoid. 
From such a double groupoid $D$ with core $C$, we can produce what we will call a \emph{Lie hopfoid} structure on $D \rightrightarrows C$ in the category $\SRel$ of smooth manifolds and smooth relations, defined via the same data and diagrams as that for a symplectic hopfoid.\footnote{In this case we require that the target and source be surmersions, that the unit be a cosurmersion, and that the antipode be a diffeomorphism.}

Indeed, suppose that $D$ is a double Lie groupoid. The description of the leaves of Lemma \ref{lem:fol} still defines a foliation on $\wt s_{H}^{-1}(1^V M)$ whose leaf space may be identified with the core $C$ in the same manner. Similarly, the core can also be realized as the leaf space of certain foliations on
\[
\wt s_V^{-1}(1^H M),\ \wt t_H^{-1}(1^V M), \text{ and } \wt t_V^{-1}(1^H M),
\]
given by the same expressions as those in the symplectic case. Then all the constructions of the previous sections produce the Lie hopfoid structure $D \rightrightarrows C$ we are interested in. For example, the ``target'' $\mathbf{t}_D: D \to C$ obtained by realizing $C$ as the leaf space of the correct foliation on $\wt s_V^{-1}(1^H M)$ is the smooth relation
\[
\mathbf{t}_D: a \mapsto a\,\wt\circ_V\,\wt 1^H_{i_V(\wt s_H(a))} \text{ for } a \in \wt s_V^{-1}(1^H M),
\]
which is the analog of equation~(\ref{target}) in the symplectic case.

\begin{rmk}
We should note that although the constructions of the previous section now work more generally in the double Lie groupoid case, the expressions for the various foliations used were discovered only through a careful study of the symplectic case where these foliations are the characteristic foliations of certain coisotropic submanifolds.
\end{rmk}

These two ``hopfoid'' structures are well-behaved with respect to the cotangent functor in the sense of the following theorem, which implies commutativity of the diagram:
\begin{center}
\begin{tikzpicture}[>=angle 90]
	\node (UL) at (0,1) {double Lie groupoids};
	\node (UR) at (6,1) {symplectic double groupoids};
	\node (LL) at (0,-1) {Lie hopfoids};
	\node (LR) at (6,-1) {symplectic hopfoids};

	\tikzset{font=\scriptsize};
	\draw[->] (UL) to node [above] {$T^*$} (UR);
	\draw[->] (UL) to node [left] {hopf} (LL);
	\draw[->] (UR) to node [right] {hopf} (LR);
	\draw[->] (LL) to node [above] {$T^*$} (LR);
\end{tikzpicture}
\end{center}
where ``hopf'' denotes the map sending double groupoids to hopfoids.

\begin{thrm}\label{thrm:lie-hopf}
Suppose that $D$ is a double Lie groupoid with core $C$ and equip $T^*D$ with the induced symplectic double groupoid structure of Theorem \ref{thrm:ctdbl} with core $T^*C$. Then the induced symplectic hopfoid structure on $T^*D \rightrightarrows T^*C$ is the one obtained by applying the cotangent functor to the Lie hopfoid $D \rightrightarrows C$ arising from $D$.
\end{thrm}

For instance, consider the double groupoid of Example \ref{ex:dmain}, and the corresponding symplectic double groupoid of Example \ref{ex:main}:
\begin{center}
\begin{tikzpicture}[>=angle 90]
	\node (LL) at (0,0) {$M$};
	\node (LR) at (2,0) {$M.$};
	\node (UL) at (0,2) {$G$};
	\node (UR) at (2,2) {$G$};
	\node (M1) at (3,1) {$$};
	
	\node (M2) at (5,1) {$$};
	\node (LL) at (6,0) {$A^*$};
	\node (LR) at (8,0) {$M$};
	\node (UL) at (6,2) {$T^*G$};
	\node (UR) at (8,2) {$G$};

	\tikzset{font=\scriptsize};
	\draw[->] (.4,.07) -- (1.6,.07);
	\draw[->] (.4,-.07) -- (1.6,-.07);
	
	\draw[->] (-.07,1.6) -- (-.07,.4);
	\draw[->] (.07,1.6) -- (.07,.4);
	
	\draw[->] (.4,2.07) -- (1.7,2.07);
	\draw[->] (.4,2-.07) -- (1.7,2-.07);
	
	\draw[->] (2-.07,1.6) -- (2-.07,.4);
	\draw[->] (2+.07,1.6) -- (2+.07,.4);
	
	\draw[->] (6+.4,.07) -- (8-.4,.07);
	\draw[->] (6+.4,-.07) -- (8-.4,-.07);
	
	\draw[->] (6+-.07,1.6) -- (6+-.07,.4);
	\draw[->] (6+.07,1.6) -- (6+.07,.4);
	
	\draw[->] (6+.6,2.07) -- (8-.3,2.07);
	\draw[->] (6+.6,2-.07) -- (8-.3,2-.07);
	
	\draw[->] (8-.07,1.6) -- (8-.07,.4);
	\draw[->] (8+.07,1.6) -- (8+.07,.4);
	
	\draw[->] (M1) to node [above] {$T^*$} (M2);
\end{tikzpicture}
\end{center}
The Lie hopfoid corresponding to the former turns out to be the groupoid $G \rightrightarrows M$ itself, where the structure maps are now viewed as smooth relations, and the symplectic hopfoid corresponding to the latter is $T^*G \rightrightarrows T^*M$, so the result above generalizes the fact that the latter hopfoid is obtained by applying the cotangent functor to the former.

\begin{proof}
First we show that the cotangent lift $T^*(\mathbf{t}_D)$ of the target relation $\mathbf{t}_D$ of $D \rightrightarrows C$ is the target relation of the symplectic hopfoid $T^*D \rightrightarrows T^*C$ constructed from the symplectic double groupoid $T^*D$. Set $X := \wt s_V^{-1}(1^H M)$, let $\pi: X \to C$ be the surjective submersion
\[
a \mapsto a\,\wt\circ_V\,\wt 1^H_{i_V(\wt s_H(a))},
\]
and let $j: X \to D$ be the inclusion. Viewing both of these as relations, the target relation $\mathbf{t}_D: D \to C$ is then the composition
\begin{center}
\begin{tikzpicture}[>=angle 90]
	\node (1) at (0,1) {$D$};
	\node (2) at (2,1) {$X$};
	\node (3) at (4,1) {$C.$};

	\tikzset{font=\scriptsize};
	\draw[->] (1) to node [above] {$j^t$} (2);
	\draw[->] (2) to node [above] {$\pi$} (3);
\end{tikzpicture}
\end{center}
Hence the cotangent lift $T^*(\mathbf{t}_D): T^*D \to T^*C$ is the composition (of reductions)
\begin{center}
\begin{tikzpicture}[>=angle 90]
	\node (1) at (0,1) {$T^*D$};
	\node (2) at (2,1) {$T^*X$};
	\node (3) at (4,1) {$T^*C,$};

	\tikzset{font=\scriptsize};
	\draw[->>] (1) to node [above] {$red$} (2);
	\draw[->>] (2) to node [above] {$red_\pi$} (3);
\end{tikzpicture}
\end{center}
where the first reduction is the one obtained by reducing the coisotropic submanifold $T^*D|_X$ of $T^*D$ and the second is the one obtained by reducing the coisotropic $N^*\F_\pi$ of $T^*X$, where $\F_\pi$ is the foliation on $X$ given by the fibers of $\pi$. Note that the domain of $T^*(\mathbf{t}_D)$ is simply $N^*\F_\pi$.

To show that this composition is the target $\mathbf{t}: T^*D \red T^*C$ of the symplectic hopfoid $T^*D \rightrightarrows T^*C$, we compute the domain of the latter. Recall that the domain of this canonical relation consists of those $(a,\xi) \in T^*D$ which map under the source $\alpha: T^*D \to A^*H$ of the left groupoid structure on $T^*D$ to a unit of the bottom groupoid structure. Identifying $A^*H$ with $N^*H$ in $T^*D$, a calculation using the explicit descriptions of $\alpha$ and of the unit $A^*C \to A^*H$ of the bottom groupoid shows that $\dom\mathbf{t}$ then consists of those $(a,\xi) \in T^*D$ such that $a \in X$ and
\[
(dL_a)_{\wt 1^V 1^H_m}^*\left(\xi|_{\ker(d\wt t_V)_a}\right) \text{ restricted to $V$ is zero},
\]
where $\wt s_V(a) = 1^H_m$, $L_a$ is left multiplication by $a$ in the left groupoid structure on $T^*D$, and by ``restricted to $V$'' we mean restricted to the elements of $\ker(d\wt t_V)_{\wt 1^V1^H_m}$ coming from $TV \subset TD$ under the decomposition
\[
TD|_H = TH \oplus \ker(d\wt t_V)|_H.
\]
Explicitly, such elements are of the form
\[
(d\wt 1^H)_{1^V_m}u - (d\wt 1^V)_{1^H_m}(d\wt t_V)_{\wt 1^H 1^V_m}(d\wt 1^H)_{1^V_m}u \text{ for } u \in TV,
\]
which, using $\wt t_V \circ \wt 1^H = 1^H \circ t_V$, we can write as
\[
(d\wt 1^H)_{1^V_m}\left[u - (d 1^V)_m(d t_V)_{1^V_m}u\right].
\]
The expression in square brackets above gives all of $\ker(d t_V)_{1^Vm}$ when varying $u$, and so finally we find that the condition on $\xi$ is that
\[
\xi \text{ vanishes on } (dL_a)_{\wt 1^V 1^H_m}(d\wt 1^H)_{1^V_m}(\ker(d t_V)_{1^V_m}).
\]
Hence $(a,\xi) \in T^*D$ is in $\dom\mathbf{t}$ if and only if $a \in X$ and $\xi$ satisfies the above condition.

According to the analog of equation~(\ref{eq:tang-dist}) in the general double groupoid case giving a description of the tangent distribution to $\F_\pi$, we see that $(a,\xi) \in T^*D$ is in the domain of $\mathbf{t}$ if and only if $(a,\xi)$ is in $N^*\F_\pi$. Thus the canonical relations $T^*(\mathbf{t}_D)$ and $\mathbf{t}$ have the same domain, and since both are reductions we conclude that $T^*(\mathbf{t}_D) = \mathbf{t}$.

A similar computation shows that the cotangent lifts of the source $\mathbf{s}_D: D \to C$ and unit $\mathbf{e}_D: C \to D$ of $D \rightrightarrows C$ are the source and unit of $T^*D \rightrightarrows T^*C$ respectively. From the definition of the symplectic groupoid structure on the cotangent bundle of a groupoid it follows that the product, coproduct, and antipode of the symplectic hopfoid $T^*D \rightrightarrows T^*C$ are indeed the cotangent lifts of the product, coproduct, and antipode of the Lie hopfoid $D \rightrightarrows C$, and the theorem is proved.
\end{proof}

\begin{ex}\label{ex:dinertia-hopf}
Consider the double groupoid of Example~\ref{ex:dinertia}. The resulting Lie hopfoid structure $G \times G \rightrightarrows G$ is the following. First, the target and source relations $G \times G \to G$ are respectively given by
\[
(g,h) \mapsto gh^{-1} \text{ for $g,h \in G$ such that $r(g)=r(h)$}
\]
and
\[
(g,h) \mapsto h^{-1}g \text{ for $g,h \in G$ such that $\ell(g)=\ell(h)$}.
\]
The unit relation $G \to G \times G$ and antipode $G \to G$ are respectively $g \mapsto \{g\} \times M$ and $(g,h) \mapsto (h^{-1},g^{-1})$. Finally, the product $G \times G \times G \times G \to G \times G$ is
\[
(g,h,a,b) \mapsto (gh,ab) \text{ for composable $g,h$ and composable $a,b$}
\]
and the coproduct $G \times G \to G \times G \times G \times G$ is
\[
(g,h) \mapsto (g,k,k,h) \text{ where $k$ is any element of $G$}.
\]

According to the theorem above, the symplectic hopfoid corresponding to the symplectic double groupoid of Example~\ref{ex:inertia} is the cotangent lift of this Lie hopfoid.\footnote{These relations turn out to  be the structure maps of the so-called \emph{inertia groupoid} of $G \rightrightarrows M$ in disguise.}
\end{ex}

\subsection{From Hopfoids to Double Groupoids}
Suppose that $D \rightrightarrows C$ is a symplectic hopfoid. In this section we show how to recover from this data a symplectic double groupoid structure on $D$.

As a guide to how this procedure works, let us first consider the case where $C = pt$. In this case, the counit $\varepsilon: D \to pt$ of the comonoid structure on $D$ is given by a Lagrangian submanifold $V$ of $D$. Denote by $H$ the Lagrangian submanifold of $D$ given by the unit $\mathbf{e}: pt \to D$. It follows from the calculations below that $(D,\mathbf{m},\mathbf{e})$ is then a strongly positive symplectic $*$-monoid, so that we now have two symplectic groupoid structures on $D$; one with base $V$ and one with base $H$. These two structures satisfy the appropriate compatibility, so that $D$ becomes a symplectic double group.

In general we proceed as follows. First, the given strongly positive $*$-comonoid structure on $D$ gives rise to a symplectic groupoid $D \rightrightarrows V$, where $V$ is the Lagrangian submanifold of $D$ representing the counit $\varepsilon_D: D \to pt$. Keeping with our notation for symplectic double groupoids, we will denote the structure maps of this symplectic groupoid as
\[
\wt t_{H}, \wt s_{H}, \wt 1^{H}, \text{etc.}
\]

Now, denote the image of the composition
\begin{center}
\begin{tikzpicture}[>=angle 90]
	\node (UL) at (0,1) {$pt$};
	\node (UM) at (2,1) {$D$};
	\node (UR) at (4,1) {$C$};
	\node (URR) at (6,1) {$D$};

	\tikzset{font=\scriptsize};
	\draw[->] (UL) to node [above] {$\varepsilon^t$} (UM);
	\draw[->] (UM) to node [above] {$\mathbf{t}$} (UR);
	\draw[->] (UR) to node [above] {$\mathbf{e}$} (URR);
\end{tikzpicture}
\end{center}
by $H$; this is a Lagrangian submanifold of $D$ since the above composition is strongly transversal given that $\mathbf{e}$ is a coreduction. (Note that since $\mathbf{i}$ preserves the counit of $D$ and $\mathbf{s} = \mathbf{t} \circ \mathbf{i}$, it follows that the composition above is the same as $\mathbf{e} \circ \mathbf{s} \circ \varepsilon^t$.) As the notation suggests, this will form the base of the second symplectic groupoid structure on $D$.

\begin{prop}
$(D,\mathbf{m},e)$, where $\mathbf{m}$ is the product of the symplectic hopfoid structure and $e: pt \to D$ is the morphism given by $H$, is a symplectic monoid.
\end{prop}

\begin{proof}
By assumption, $\mathbf{m}$ is a associative. The unit properties of $e$ follow from the diagrams in condition (\ref{diag:units}) in the definition of a symplectic hopfoid as follows. The first diagram looks like:
\begin{center}
\begin{tikzpicture}[>=angle 90]
	\node (UL) at (0,1) {$D \times D$};
	\node (UM) at (2.5,1) {$C \times D$};
	\node (UR) at (5,1) {$D \times D$};
	\node (LL) at (0,-1) {$D$};
	\node (LR) at (5,-1) {$D.$};

	\tikzset{font=\scriptsize};
	\draw[->] (UL) to node [above] {$\mathbf{t} \times id$} (UM);
	\draw[->] (UM) to node [above] {$\mathbf{e} \times id$} (UR);
	\draw[->] (LL) to node [left] {$\Delta$} (UL);
	\draw[->] (UR) to node [right] {$\mathbf{m}$} (LR);
	\draw[->] (LL) to node [above] {$id$} (LR);
\end{tikzpicture}
\end{center}
For any $a \in D$ the canonical relation $\Delta$ in particular contains an element of the form $(a,\wt 1^{H}_\lambda,a)$, which follows from the fact that $\Delta^t$ is in fact a groupoid product. This element is in the canonical relation
\[
\varepsilon_D^t \times id: D \to D \times D,
\]
and thus the composition $(\mathbf{t} \times id) \circ \Delta$ contains $(\mathbf{t} \times id) \circ (\varepsilon_D^t \times id) = (\mathbf{t} \circ \varepsilon_D^t) \times id$. Thus the commutativity of the diagram above implies the commutativity of
\begin{center}
\begin{tikzpicture}[>=angle 90]
	\node (UL) at (0,1) {$C \times D$};
	\node (UR) at (5,1) {$D \times D$};
	\node (LL) at (0,-1) {$D$};
	\node (LR) at (5,-1) {$D.$};

	\tikzset{font=\scriptsize};
	\draw[->] (UL) to node [above] {$\mathbf{e} \times id$} (UR);
	\draw[->] (LL) to node [left] {$(\mathbf{t} \circ \varepsilon_D^t) \times id$} (UL);
	\draw[->] (UR) to node [right] {$\mathbf{m}$} (LR);
	\draw[->] (LL) to node [above] {$id$} (LR);
\end{tikzpicture}
\end{center}
Recalling the definition of $e = \mathbf{e} \circ \mathbf{t} \circ \varepsilon_D^t$, this then says that $e$ is a left unit for $\mathbf{m}$. The right unit property of $e$ follows similarly from the second diagram in condition (vii), and we thus conclude that $(D,\mathbf{m},e)$ is a symplectic monoid.
\end{proof}

Now, it is simple to check that $\mathbf{i} \circ \wt i_{H}$ is a $*$-structure on the above monoid using condition~(\ref{diag:antipode}) in the definition of a symplectic hopfoid and the corresponding $*$-properties of the $*$-structure on $(D,\Delta_D,\varepsilon_D)$.

\begin{prop}
The $*$-structure above is strongly positive.
\end{prop}

\begin{proof}
Consider the first diagram in condition~(\ref{diag:inverses}) in the definition of symplectic hopfoid. It is a simple check to see that the following then commutes:
\begin{center}
\begin{tikzpicture}[>=angle 90]
	\node (P) at (-2,-2) {$pt$};
	\node (UL) at (0,1) {$D \times D$};
	\node (UR) at (5,1) {$D \times D$};
	\node (LL) at (0,-1) {$D$};
	\node (LM) at (2.5,-1) {$C$};
	\node (LR) at (5,-1) {$D$};

	\tikzset{font=\scriptsize};
	\draw[->] (P) to node [above] {$\varepsilon_D^t$} (LL);
	\draw[->] (P) to [bend left=15] node [left] {$\wt i_{H}$} (UL);
	\draw[->] (P) to [bend right=15] node [above] {$$} (LM);
	\draw[->] (UL) to node [above] {$id \times \mathbf{i}$} (UR);
	\draw[->] (LL) to node [left] {$\Delta$} (UL);
	\draw[->] (UR) to node [right] {$\mathbf{m}$} (LR);
	\draw[->] (LL) to node [above] {$\mathbf{t}$} (LM);
	\draw[->] (LM) to node [above] {$\mathbf{e}$} (LR);
	\draw[->] (P) to [bend right=15] node [below] {$e$} (LR);
\end{tikzpicture}
\end{center}
where $\wt i_{H}: \overline{D} \to D$ is now thought of as a morphism $pt \to D \times D$ via its graph. Now, we can factor the top row as
\begin{center}
\begin{tikzpicture}[>=angle 90]
	\node (UL) at (0,1) {$D \times D$};
	\node (UM) at (3,1) {$D \times \overline{D}$};
	\node (UR) at (7,1) {$D \times D$};

	\tikzset{font=\scriptsize};
	\draw[->] (UL) to node [above] {$id \times \wt i_{H}$} (UM);
	\draw[->] (UM) to node [above] {$id \times (\mathbf{i} \circ \wt i_{H})$} (UR);
\end{tikzpicture}
\end{center}
using the fact that $\wt i_{H} \circ \wt i_{H} = id$. Then the composition
\begin{center}
\begin{tikzpicture}[>=angle 90]
	\node (UL) at (0,1) {$pt$};
	\node (UM) at (3,1) {$D \times D$};
	\node (UR) at (6,1) {$D \times \overline{D}$};

	\tikzset{font=\scriptsize};
	\draw[->] (UL) to node [above] {$\wt i_{H}$} (UM);
	\draw[->] (UM) to node [above] {$id \times \wt i_{H}$} (UR);
\end{tikzpicture}
\end{center}
is precisely the morphism $pt \to D \times \overline{D}$ given by the diagonal of $D \times \overline{D}$. Thus the diagram above becomes
\begin{center}
\begin{tikzpicture}[>=angle 90]
	\node (UL) at (0,1) {$D \times \overline{D}$};
	\node (UR) at (4,1) {$D \times D$};
	\node (LL) at (0,-1) {$pt$};
	\node (LR) at (4,-1) {$D$};

	\tikzset{font=\scriptsize};
	\draw[->] (UL) to node [above] {$id \times (\mathbf{i} \circ \wt i_{H})$} (UR);
	\draw[->] (LL) to node [left] {$$} (UL);
	\draw[->] (UR) to node [right] {$\mathbf{m}$} (LR);
	\draw[->] (LL) to node [above] {$e$} (LR);
\end{tikzpicture}
\end{center}
which is precisely the strong positivity requirement on $(D,\mathbf{m},e)$.
\end{proof}

We thus conclude that the monoid $(D,\mathbf{m},e)$ produces an additional symplectic groupoid structure on $D$. Restricting the groupoid structure on $D \rightrightarrows V$ to $H$ produces a groupoid structure on $H \rightrightarrows M$, and restricting the groupoid structure on $D \rightrightarrows H$ to $V$ gives a groupoid structure $V \rightrightarrows M$.

These give the four groupoid structures in the definition of a symplectic double groupoid, and it is straightforward to check that the various compatibilities hold; in particular, condition~(\ref{diag:product}) defining a symplectic hopfoid precisely says that
\[
\mathbf{m}: D \times D \to D
\]
is a groupoid morphism for the groupoid product $\Delta^t$ on $D$ and vice-versa. We thus have:

\begin{thrm}\label{thrm:main-cor}
There is a $1$-$1$ correspondence between symplectic double groupoids and symplectic hopfoids.
\end{thrm}

\begin{proof}
It remains only to check that the operations of producing a symplectic hopfoid from a symplectic double groupoid and conversely producing a symplectic double groupoid from a symplectic hopfoid are inverse to one another. This is a straightforward verification.
\end{proof}

The theorem above has an obvious generalization to the general smooth setting, giving a $1$-$1$ correspondence between arbitrary double Lie groupoids and Lie hopfoids.

\subsection{Morphisms}
Our results on the correspondence between double groupoids and hopfoids are only concerned with these structures at the level of objects. We finish by briefly and speculatively considering how to bring morphisms into the picture.

There is a natural notion of morphism between symplectic hopfoids:
\begin{defn}
Suppose $D \rightrightarrows C$ and $D' \rightrightarrows C'$ are symplectic hopfoids. A \emph{morhpism} $D \to D'$ is a pair $(L,L_c)$ of canonical relations $L: D \to D'$ and $L_c: C \to C'$ such that $L$ is a comonoid morphism for the comonoid structures on $D$ and $D'$ and which are compatible with the structure morphisms of $D$ and $D'$ in the appropriate sense. For instance, the compatibility between $(L,L_c)$ and the target morphisms of $D$ and $D'$, and that between $L$ and the product morphisms of $D$ and $D'$, are expressed by the commutativity of:
\begin{center}
\begin{tikzpicture}[>=angle 90]
	\node (UL) at (0,1) {$D$};
	\node (UR) at (3,1) {$D'$};
	\node (LL) at (0,-1) {$C$};
	\node (LR) at (3,-1) {$C'$};
	
	\node (M) at (4.5,0) {and};
	
	\node (UL2) at (6,1) {$D\times D$};
	\node (UR2) at (9,1) {$D' \times D'$};
	\node (LL2) at (6,-1) {$D$};
	\node (LR2) at (9,-1) {$D'$};

	\tikzset{font=\scriptsize};
	\draw[->] (UL) to node [above] {$L$} (UR);
	\draw[->] (UL) to node [left] {$\mathbf t$} (LL);
	\draw[->] (UR) to node [right] {$\mathbf t'$} (LR);
	\draw[->] (LL) to node [above] {$L_c$} (LR);
	
	\draw[->] (UL2) to node [above] {$L \times L$} (UR2);
	\draw[->] (UL2) to node [left] {$\mathbf m$} (LL2);
	\draw[->] (UR2) to node [right] {$\mathbf m'$} (LR2);
	\draw[->] (LL2) to node [above] {$L$} (LR2);
\end{tikzpicture}
\end{center}
respectively. We omit writing out all explicit compatibility diagrams here, but note that all compositions arising in such diagrams should be assumed to be strongly transversal.
\end{defn}

In the case where $C = pt$, so that symplectic hopfoids are simply Hopf algebra objects in $\Symp$, this definition reduces to that of a morphism between Hopf algebra objects as given in Definition~\ref{defn:hopf-morphism}, which suggests this is a good notion of morphism to consider.

As for morphisms between symplectic double groupoids, we consider the following. Given two double Lie groupoids $D$ and $D'$, one can take a morphism $f: D \to D'$ to be simply a smooth map which is a homomorphism for both the horizontal and vertical groupoid structures on $D$ and $D'$. The cotangent lift of this map gives a canonical relation $T^*f: T^*D \to T^*D'$ which, viewed as a subset of $\overline{T^*D} \times T^*D'$, is a Lagrangian subgroupoid of both symplectic groupoid structures on the product $\overline{T^*D} \times T^*D'$ induced by the horizontal and vertical structures in Theorem \ref{thrm:ctdbl}. More generally, we can take a morphism $D \to D'$ to be a submanifold of $D \times D'$ which is a Lie subgroupoid for both product structures induced by the horizontal and vertical groupoid structures on $D$ and $D'$. The cotangent lift of such a subgroupoid again gives a Lagrangian subgroupoid of $\overline{T^*D } \times T^*D'$ for both product structures, which thus suggests the following definition:

\begin{defn}
Suppose $D$ and $D'$ are symplectic double groupoids. A \emph{morphism} $D \to D'$ is a canonical relation $L: D \to D'$ which is a Lagrangian subgroupoid of $\overline D \times D'$ for each of the product structures induced by the horizontal and vertical groupoid structures on $D$ and $D'$.
\end{defn}

Given such a canonical relation $L: D \to D'$, the reduction procedure of Section~\ref{sec:reduc} will, under certain transversality assumptions, produce a canonical relation $L_c: C \to C'$ (where we view the cores $C$ and $C'$ of $D$ and $D'$ respectively as appropriate leaf spaces) via standard methods of producing Lagrangian submanifolds of symplectic quotients. To be precise: $L_c$, as a Lagrangian submanifold of $\overline C \times C'$, is defined to be the result of the composition
\begin{center}
\begin{tikzpicture}[>=angle 90]
	\node (UL) at (0,1) {$pt$};
	\node (UM) at (3,1) {$\overline D \times D'$};
	\node (UR) at (6,1) {$\overline C \times C'$};

	\tikzset{font=\scriptsize};
	\draw[->] (UL) to node [above] {$L$} (UM);
	\draw[->] (UM) to node [above] {$red \times red$} (UR);
\end{tikzpicture}
\end{center}
in the case where this composition is strongly transversal. We thus obtain a pair $(L,L_c)$ of canonical relations $L: D \to D'$ and $L_c: C \to C'$, and we conjecture that these do form the data of a morphism between the symplectic hopfoids $D \rightrightarrows C$ and $D' \rightrightarrows C'$. For instance, if we take the reductions $red$ in the composition above to be the target morphisms $\mathbf t$ and $\mathbf t'$, this definition of $L_c$ makes the fact that $(L,L_c)$ is compatible with the target morphisms of $D$ and $D'$ immediate. Going the other way, it should be possible to recover a morphism $D \to D'$ between double groupoids given a morphism between hopfoids, so that Theorem~\ref{thrm:main-cor} holds at the level of morphisms as well.

The verification of all these details, which involves checking the commutativity of various diagrams, is worthy of its own work and so will not be carried out here. Once this has been established, the result that the commutativity of
\begin{center}
\begin{tikzpicture}[>=angle 90]
	\node (UL) at (0,1) {double Lie groupoids};
	\node (UR) at (6,1) {symplectic double groupoids};
	\node (LL) at (0,-1) {Lie hopfoids};
	\node (LR) at (6,-1) {symplectic hopfoids};

	\tikzset{font=\scriptsize};
	\draw[->] (UL) to node [above] {$T^*$} (UR);
	\draw[->] (UL) to node [left] {hopf} (LL);
	\draw[->] (UR) to node [right] {hopf} (LR);
	\draw[->] (LL) to node [above] {$T^*$} (LR);
\end{tikzpicture}
\end{center}
also holds at the level of morphisms, where we define morphisms between double Lie groupoids as Lie subgroupoids of products, should be a consequence of the functorial properties of $T^*$. Another interesting possibility would be to take ``symplectic bibundles'' as morphisms between symplectic double groupoids, but it is not clear how this notion should be translated over to hopfoids.


\begin{thebibliography}{9}

\bibitem{BM} R. Brown and K. C. H. Mackenzie, \emph{Determination of a Double Lie Groupoid by its Core Diagram}, J. Pure Appl. Algebra 80: 237-272, 1992

\bibitem{CC} A. Cattaneo and I. Contreras, \emph{Relational Symplectic Groupoids}, Letters in Mathematical Physics 105(5): 723-767, 2015

\bibitem{CDW2} A. Cattaneo, B. Dherin, and A. Weinstein, \emph{Symplectic Microgeometry I: Micromorphisms}, Journal of Symplectic Geometry 8(2): 205-223, 2010

\bibitem{CDW} A. Coste, P. Dazord, and A. Weinstein, \emph{Groupo\"ides Symplectiques}, Publications du D\'epartment de math\'ematiques (Lyon): 1-62, 1987

\bibitem{H} R. Hepworth, \emph{Vector fields and flows on differential stacks}, Theory and Applications of Categories 22(21): 542-587, 2009

\bibitem{LW} D. Li-Bland and A. Weinstein, \emph{Selective  Categories and Linear Canonical Relations}, SIGMA 10(100), 2014

\bibitem{M} K. C. H. Mackenzie, \emph{On Symplectic Double Groupoids and the Duality of Poisson Groupoids}, International Journal of Mathematics 10(4): 435-456, 1999

\bibitem{MT} R. A. Mehta and X. Tang, \emph{From double Lie groupoids to local Lie 2-groupoids}, Bulletin of the Brazilian Mathematical Society 42(4): 651-681

\bibitem{SZ} I. Szymczak and S. Zakrzewski, \emph{Quantum deformations of the {H}eisenberg group obtained by geometric quantization}, Journal of Geometry and Physics 7(4): 553-569

\bibitem{SZ1} S. Zakrzewski, \emph{Quantum and Classical Pseudogroups. {I}. {U}nion Pseudogroups and their Quantization}, Comm. Math. Phys. 134(2): 347-370, 1990

\bibitem{SZ2} S. Zakrzewski, \emph{Quantum and Classical Pseudogroups. {I}{I}. {D}ifferential and Symplectic Pseudogroups}, Comm. Math. Phys. 134(2): 371-395, 1990

\bibitem{W} A. Weinstein, \emph{The volume of a differentiable stack}, Lett Math Phys 90:353-371, 2009

\bibitem{W1} A. Weinstein, \emph{Coisotropic calculus and Poisson groupoids}, Journal of the Mathematical Society of Japan 40(4): 705-727, 1988

\bibitem{W2} A. Weinstein, \emph{A note on the Wehrheim-Woodward category}, Journal of Geometric Mechanics 3(4): 507-515, 2011

\bibitem{W3} A. Weinstein, \emph{The symplectic category}, Differential Geometric Methods in Mathematical Physics, Lecture Notes in Mathematics Vol 905, Springer, Berlin New York, 1982

\end{thebibliography}
\end{document}